\documentclass[11pt,a4paper]{article}
\usepackage{pdfsync} 

\usepackage{amssymb,amsmath,amsfonts,amsthm} 
\newtheorem{theorem}{Theorem}
\numberwithin{theorem}{section}
\newtheorem{lemma}{Lemma}
\newtheorem{corollary}{Corollary}

\usepackage{amssymb}
\usepackage{hyperref}
\usepackage{graphicx}
\usepackage{wrapfig}
\usepackage{subfig}
\usepackage{epsfig}
\usepackage{float,mathrsfs}
\usepackage{enumitem}
\usepackage{dsfont}

\usepackage{bbm}
\usepackage{algorithm,algpseudocode}

\newtheorem{remark}[theorem]{Remark}
\newtheorem{assumption}[theorem]{Assumption}

\usepackage{color}

\definecolor{refkey}{rgb}{0.9451,0.2706,0.4941}\definecolor{labelkey}{rgb}{0.9451,0.2706,0.4941}

\definecolor{darkred}{RGB}{139,0,0}
\definecolor{darkgreen}{RGB}{0,100,0}
\definecolor{darkmagenta}{RGB}{139,0,139}

\newcommand{\bsx}{{\boldsymbol{x}}}
\newcommand{\bsm}{{\boldsymbol{m}}}

\newcommand{\bsy}{{\boldsymbol{y}}}

\newcommand{\bsnu}{{\boldsymbol{\nu}}}

\newcommand{\bsb}{{\boldsymbol{b}}}

\newcommand{\bsk}{{\boldsymbol{k}}}

\newcommand{\calA}{\mathcal{A}}
\newcommand{\calB}{\mathcal{B}}

\newcommand{\cJ}{\mathcal{J}}

\newcommand{\calF}{\mathcal{F}}
\newcommand{\calQ}{\mathcal{Q}}

\newcommand{\calX}{\mathcal{X}}
\newcommand{\calY}{\mathcal{Y}}
\newcommand{\calZ}{\mathcal{Z}}

\newcommand{\frakJ}{\mathfrak{J}}

\newcommand{\E}{\mathbb{E}}

\newcommand{\I}{\mathbb{I}}

\newcommand{\N}{\mathbb{N}}

\newcommand{\R}{\mathbb{R}}

\newcommand{\bsrho}{\boldsymbol{\rho}}
\newcommand{\Cov}{\mathrm{Cov}}

\DeclareMathOperator{\Tr}{Tr}
\DeclareMathOperator{\diag}{diag}

\DeclareMathOperator*{\argmin}{arg\,min}

\newcommand{\rev}[1]{\textcolor{black}{#1}}

\title{Application of dimension truncation error analysis to high-dimensional function approximation}

\title{One-shot Learning of Surrogates in PDE-constrained Optimization Under Uncertainty}
\author{Philipp A. Guth\thanks{Johann Radon Institute for Computational and Applied Mathematics, Austrian Academy of Sciences, Altenbergerstra{\ss}e 69, 4040 Linz, Austria
	({\tt philipp.guth@ricam.oeaw.ac.at}).}
\and Claudia Schillings\thanks{Fachbereich Mathematik und Informatik, Freie Universität Berlin, Arnimallee 6, 14195 Berlin, Germany 
	({\tt c.schillings@fu-berlin.de}).}
\and Simon Weissmann\thanks{Institute of Mathematics, University of Mannheim, 68138 Mannheim, Germany
  ({\tt simon.weissmann@uni-mannheim.de}).}}

\begin{document}
\maketitle

\renewcommand{\thefootnote}{\arabic{footnote}}

\begin{abstract}
We propose a general framework for machine learning based optimization under uncertainty. Our approach replaces the complex forward model by a surrogate, which is learned simultaneously in a one-shot sense when solving the optimal control problem. Our approach relies on a reformulation of the problem as a penalized empirical risk minimization problem for which we provide a consistency analysis in terms of large data and increasing penalty parameter. To solve the resulting problem, we suggest a stochastic gradient method with adaptive control of the penalty parameter and prove convergence under suitable assumptions on the surrogate model. Numerical experiments illustrate the results for linear and nonlinear surrogate models.
\end{abstract}
\noindent {\footnotesize
	{\bf Keywords.} {Surrogate Learning, Optimization under Uncertainty, Uncertainty Quantification, Stochastic Gradient Descent, PDE-constrained Risk Minimization}
}


\section{Introduction}

Uncertainties have the potential to render worthless even highly sophisticated methods for large-scale inverse and optimal control problems, since their conclusions do not realize in practice due to imperfect knowledge about model correctness, data relevance, and numerous other factors that influence the resulting solutions. To quantify the uncertainty, we consider risk measures ensuring the robustness of the solutions, see \cite{Kouri2018} and the references therein. For complex processes, the incorporation of these uncertainties typically results in high or even infinite dimensional problems in terms of the uncertain parameters as well as the optimization variables, which in many cases are not solvable with current state of the art methods. One promising potential remedy to this issue lies in the approximation of the \rev{underlying model} using novel techniques arising in uncertainty quantification and machine learning. In this paper, we propose a general framework for incorporating surrogate models in optimization under uncertainty. The underlying optimal control problem is given by a partial differential equation (PDE) constrained optimization problem, where the uncertain input coefficients of the PDE are modeled as a random field. In general, we are interested in the situation, where a high resolution is needed for accurate approximations of the PDE. Hence, introducing an empirical approximation of the risk measure is attached with high computational costs as for each data point we need to solve the underlying PDE model.

Our framework is based on one-shot optimization approaches \cite{volker}, where we reformulate the constrained optimization problem as an unconstrained one via a penalization method. In a recent work \cite{GSW2020}, we have introduced neural network based one-shot approach for inverse problems, where the starting point of the formulation has been mainly based on the all-at-once approach for inverse problems \cite{K16,K17} and physics-informed neural networks \cite{RPK17,RPK17_2,RPK19,YMK20}. We established the connection between the Bayesian approach and the one-shot formulation allowing to interpret the penalization parameter as the level of model error \rev{of the underlying system}.

In the setting of this article, we consider the optimization \rev{with respect }to the control function and the corresponding optimal PDE solution in a one-shot fashion. In order to force the feasibility \rev{with respect }to the constraints, we include a penalty parameter allowing for more and more weight on the penalization term. From a Bayesian perspective, this parameter controls the model error, i.e., increasing the penalization parameter corresponds to vanishing model noise. This setting allows us to incorporate surrogate models straightforwardly. We replace the optimization \rev{with respect }to the infinite dimensional PDE solution by a parameterized family of functions, where the resulting optimization task is \rev{with respect }to the parameters describing the surrogate models. Examples of surrogate models include polynomial series representations, neural networks, Gaussian process approximations and low rank approximations. We discuss various choices in Section~\ref{sec:surrogates}. However, please note that the suggested approach is not limited to the surrogates discussed here.

\subsection{Literature Overview}
Introducing uncertainties in optimization and control problems governed by PDEs leads
to highly complex optimization tasks. The proper problem formulation, the well-posedness and the development of efficient algorithms have received increasing attention in the past, see e.g., \cite{doi:10.1137/16M106306X,AUH,cite-key,BSSW,CVG, GKKSS21,KHRB,Kouri2018,doi:10.1137/140954556, Drew2018,doi:10.1137/19M1263121,vanBarelVandewalle}. 
Despite the growing interest and recent advances in PDE-constrained optimization under uncertainty, including  uncertainty  in form of random parameters or fields is still not feasible for many PDE models due to the significant increase in the complexity of the resulting optimization or control problems. The use of surrogate models, i.e., the replacement of the expensive forward model by approximations which are usually cheap to evaluate, is a promising direction in order to reduce the overall computational effort. However, the surrogates need to be trained or calibrated in advance. In particular, in the optimization under uncertainty setting, a surrogate is needed for every feasible control. One promising remedy to this issue lies in one-shot approaches, see e.g., \cite{volker} and \cite{GSW2020, SCHILLINGS201178}. In \cite{DBLP:journals/simods/GuntherRSCG20}, one-shot ideas have been successfully generalized for the training of residual neural networks. In addition, recent results on the convergence and error analysis of machine learning techniques for high-dimensional, complex systems open up the perspective to develop novel surrogates for optimization under uncertainty that admit a mathematically rigorous convergence analysis and are applicable to a large class of computationally intense, real-world problems. Neural networks (NN) have been successfully applied to various classes of PDEs, cp.~e.g., \cite{BHKS2021,Gitta1, Han8505,Gitta2,OPS19,SZ19,Y17} and also as approximation to the underlying model \cite{dong2020optimization,LJK2019}. For parametric PDEs, generalized polynomial chaos expansion haven been extensively studied, cp.~\cite{CohenDeVore15} for an overview on approximation results. Recently, Gaussian processes haven been suggested for solving general nonlinear PDEs \cite{chen2021solving}. Here, we propose a general framework, which allows 
to include 
surrogate models in a one-shot approach.

\subsection{Outline of the paper}
In this paper, we are going to analyze the dependence of the optimization error on the number of data points as well as on the weight on the penalization. \rev{More precisely, in Section 5, we connect a constrained risk minimization problem (cRM) and the penalized risk minimization (pRM) problem using a quadratic penalty approach. Finally, the pRM is compared to a penalized empirical risk minimization problem (pERM).} Furthermore, we propose a stochastic gradient descent method in order to implement the resulting empirical risk minimization problem efficiently. In this article, we make the following contributions:
\begin{itemize}
    \item We formulate a penalized empirical risk minimization problem and provide a consistency result in terms of large data limit as well as increasing penalty parameters. To be more precise, we can split the error~\rev{into}~an error term decreasing with~\rev{the}~number of data points but depending on the penalty parameter~\rev{and}~an error term decreasing~\rev{with}~the strength of~\rev{the}~penalization independently of the number of data points. Finally, if we increase the penalty parameter for~\rev{an}~increasing number of data points, we are able to derive a rate of convergence for the mean squared error.~\rev{The assumptions we shall make for our consistency analysis can be verified for linear surrogate models.}
    \item We formulate a stochastic gradient descent method in order to solve the penalized risk minimization problem, where we allow an adaptive increase of the penalty parameter avoiding numerical instabilities due to high variance. Under suitable assumptions we prove convergence of the proposed stochastic gradient descent method. These assumptions can be verified for linear surrogate models.~\rev{Furthermore, we provide an extension to a non-convex setting, where we prove convergence of the gradients under smoothness assumption.}
    \item We test our proposed approach numerically, where we apply a linear as well as a nonlinear surrogate model. The linear surrogate model is based on a polynomial expansion, while the nonlinear surrogate model is described as a neural network.
\end{itemize}

The remainder of this article is structured as follows. In Section~\ref{sec:parametricOperatorEqns} we recall results for parametric operator equations that are important for our problem formulation.
In Section~\ref{sec:mathset} we 
formulate the underlying setting of optimization under uncertainty in reduced and all-at-once formulation. In Section~\ref{sec:surrogates} we present a collection of surrogates for the parametric mapping. A detailed consistency analysis \rev{with respect }to number of data points and penalty weight is presented in Section~\ref{sec:consistency}. We formulate and analyse the penalized stochastic gradient descent method in Section~\ref{sec:SGD}. In Section~\ref{sec:numerics} we illustrate numerically the performance of our proposed framework. We conclude the paper with a brief summary and a discussion about further directions of interest in Section~\ref{sec:conclusions}.

\section{Parametric linear operator equations}
\label{sec:parametricOperatorEqns}
Many PDE problems that are subject to uncertainty can be stated as parametric operator equation. To establish the connection between the reduced formulation and the one-shot formulation of the optimization problem under uncertainty, we recall selected results for parametric operator equations.
Let $V$ and $W$ be Banach spaces over $\R$. By $W'$ we denote the topological dual space of $W$, i.e., the space of bounded linear maps from $W$ to $\R$. Moreover, let $\mathcal{L}(V,W')$ denote the Banach space of bounded linear operators from $V$ to $W'$.
Let $U$ be a nonempty topological space. A parametric linear operator from $V$ to $W'$ with parameter domain $U$ is a continuous map
\begin{align*}
    A: U \to \mathcal{L}(V,W')\,,\quad \bsy \mapsto A(\bsy)\,.
\end{align*}
For a given $z \in W'$, we are interested in finding $u: U \to V$ such that
\begin{align}\label{parametric_op}
    A(\bsy)u(\bsy) = z\,,\quad \forall \bsy \in U\,.
\end{align}
Note that we assume here that $z$ is constant in $\bsy$, i.e., $\bsy \mapsto z(\bsy) = z$ for all $\bsy \in U$.

\begin{theorem}[Theorem 1.1.1 and Lemma 1.1.3 in \cite{Gittelson2011}]
    Let $A(\bsy)$ be bijective for all $\bsy \in U$. Then \eqref{parametric_op} has a unique solution $u:U \to V$. The solution $\bsy \mapsto u(\bsy)$ is continuous if $\bsy \mapsto z(\bsy)$ is continuous. 
    Moreover, if $U$ is a compact Hausdorff space, then there exists $a_{\min},a_{\max} >0$ such that
    \begin{align*}
        \|A(\bsy)\|_{\mathcal{L}(V,W')} \leq a_{\max}\,\quad\text{and}\quad \|A(\bsy)^{-1}\|_{\mathcal{L}(W',V)} \leq 1/a_{\min}\,\quad \forall \bsy \in U\,.
    \end{align*}
\end{theorem}
\rev{In the following we shall hence work under the assumption that $U$ is a compact Hausdorff space.}

Let $\rev{\mathcal{C}(U,V)}$ denote the Banach space of continuous maps $U \to V$ with norm $\|v\|_{\mathcal{C}(U,V)} := \sup_{\bsy \in U} \|v\rev{(\bsy)}\|_{V}$. Then, the operators
\begin{eqnarray}
&\mathcal{A}:& \mathcal{C}(U,V) \to \mathcal{C}(U,W'), \quad v\mapsto [\bsy \mapsto \rev{A(\bsy) v(\bsy)}] \label{eq:contOperator}\\
&\mathcal{A}^{-1}:& \mathcal{C}(U,W') \to \mathcal{C}(U,V), \quad w\mapsto [\bsy \mapsto \rev{A(\bsy)^{-1} w(\bsy)}]
\end{eqnarray}
are well-defined {and} inverse to \rev{each other} with norms $\|\mathcal{A}\|\leq a_{\max}$ and $\|\mathcal{A}^{-1}\|\leq 1/a_{\min}$. This result can be extended to Lebesgue spaces of vector-valued functions. To this end, let $\mathfrak{B}(U)$ be the Borel $\sigma$-algebra on $U$, and let $\mu$ be a finite measure on $(U, \mathfrak{B}(U))$.
\begin{theorem}[Theorem 1.1.6 in \cite{Gittelson2011}]
For all $1 \leq p < \infty$, the operator $\mathcal{A}$ in \eqref{eq:contOperator} extends uniquely to a boundedly invertible operator on the Lebesgue--Bochner spaces
\begin{equation}
\mathcal{A} : L^p_\mu(U,V) \to L^p_\mu(U,W')\,.
\end{equation}
The norms of $\mathcal{A}$ and $\mathcal{A}^{-1}$ are bounded by $1/a_{\min}$ and $a_{\max}$, respectively.
\end{theorem}
The norm of the Lebesgue--Bochner space $L^p_{\mu}(U,V)$ is given as $\|\cdot\|_{L^p_{\mu}(U,V)} := \big(\int_U\|\cdot\|_V^p\,\rev{\mu(\mathrm d \bsy)}\big)^{1/p}$.
\begin{corollary}[Corollary 1.1.8 in \cite{Gittelson2011}]\label{coro:1.1.8}
Let $V$ and $W$ be separable spaces, $1 \leq p < \infty$, and let $q$ be the H\"older conjugate of $p$. If $z \in L^p_\mu(U,W')$, then there is a unique $\tilde{u} \in L^p_\mu(U,V)$ such that
\begin{equation}\label{eq:weakparameter}
\int_U \langle A(\bsy) \tilde{u}(\bsy), w(\bsy) \rangle_{W',W}\,\rev{\mu(\mathrm d\bsy)} = \int_U \langle z,w(\bsy)\rangle_{W',W}\,\rev{\mu(\mathrm d\bsy)}\,,\quad \forall\, w(\bsy) \in L^q_\mu(U,W)\,,
\end{equation}
where $\langle\cdot,\cdot\rangle_{W',W}$ denotes the duality pairing of $W$ with its dual $W'$. Moreover, the solution $u\rev{(\bsy)}$ of \eqref{parametric_op} is a version of $\tilde{u}\rev{(\bsy)}$.
\end{corollary}
We thus no longer distinguish between the solution $u\rev{(\bsy)}$ of \eqref{parametric_op} and its equivalence class $\tilde{u}\rev{(\bsy)}$.


\section{PDE-constrained Optimization Under Uncertainty}\label{sec:mathset}

We are interested in solving an optimal control problem in the presence of uncertainty by minimizing the averaged least square difference of the state $u$ and a desired target state $u_0$. The state $u$ is the solution of a parametric linear operator equation, steered by a control function, and having a parameter vector as input coefficient. The parameter vector is in principle infinite-dimensional, and in practice might need a large finite number of terms for accurate approximation.

Our goal of computation is the following optimal control problem
\begin{equation}
\min_{z \in \calZ, u\in \rev{L^{2}_\mu(U,V)}} J(u,z)\,,\quad J(u,z) := \frac{1}{2} \int_U \|\mathcal{Q}u-u_0\|^2_{\frakJ}\,\rev{\mu(\mathrm d\bsy)} + \frac{\alpha}{2} \|z\|^2_{\mathcal{Z}}\,,
\label{eq:1.1}
\end{equation}
subject to the parametric linear operator equation in $L^{\rev{2}}_\mu(U,W')$
\begin{align}\label{eq:constraint}
    \mathcal{A}u = \mathcal{B}z\,,
\end{align}
\rev{where $\calZ$ and $\frakJ$ are Hilbert spaces}, $u_0 \in \frakJ$, $\mathcal{Q} \in \mathcal{L}(V,\frakJ)$, $\mathcal{B}\in \mathcal{L}(\calZ,W')$. In particular, the operators $\mathcal{B}$ and $\mathcal{Q}$ are not dependent on $\bsy$ and thus can be uniformly bounded for all $\bsy$, i.e., $\|\mathcal{B}\|_{\mathcal{L}(\calZ,W')}\leq C_1$ and $\|\mathcal{Q}\|_{\mathcal{L}(V,\frakJ)}\leq C_2$ for some $C_1,C_2>0$ and all $\bsy \in U$. This implies in particular, that $\mathcal{B}z\in L^{\rev{2}}_\mu(U,W')$ for all deterministic controls $z\in \calZ$ and $\mathcal{Q}u \in L^{2}_\mu(U,\frakJ)$ for all $u \in L^{2}_\mu(U,V)$. Moreover, we assume that $\calQ$ and $\calB$ have a bounded inverse, i.e., $\|\calB^{-1}\|_{\mathcal{L}(W',\calZ)}\leq C_3$ and $\|\mathcal{Q}^{-1}\|_{\mathcal{L}(\frakJ,V)}\leq C_4$ for $C_3,C_4>0$. In particular, we assume that $\mathcal{A}$ is a boundedly invertible operator as described in Section \ref{sec:parametricOperatorEqns}.

\begin{theorem}
    Let $\alpha \rev{~>~} 0$. Then problem \eqref{eq:1.1} -- \eqref{eq:constraint} has a \rev{unique} solution $(z^*,u^*)$. 
\end{theorem}
\begin{proof}
    We note, that $L^2_{\mu}(U,V)$ is a Hilbert space if $V$ is a Hilbert space. The result then follows from standard optimal control theory, see, e.g., \cite[Theorem 1.43]{hinze2008optimization}.
\end{proof}

Substituting $u = \mathcal{A}^{-1}\mathcal{B}z$ into $J$ gives $\hat J(z) := J(\mathcal{A}^{-1}\mathcal{B}z,z)$ and leads to the equivalent formulation of problem \eqref{eq:1.1} -- \eqref{eq:constraint}
\begin{equation}\label{eq:reduced}
    \min_{z \in \rev{\mathcal{Z}}} \hat J(z)\,,\quad \hat J(z) := \frac{1}{2} \int_U \|\mathcal{Q}\mathcal{A}^{-1}\mathcal{B}z-u_0\|^2_{\frakJ}\,\rev{\mu(\mathrm d\bsy)} + \frac{\alpha}{2} \|z\|^2_{\calZ}\,.
\end{equation}

\rev{
\begin{remark}
From Corollary \ref{coro:1.1.8} we conclude that the solution $u$ of \eqref{eq:constraint} and the solution $\hat{u}$ of 
\begin{align}\label{eq:as_parametric_op}
A(\bsy) \hat{u}(\bsy) = \mathcal{B} z,\quad \text{for all } \bsy \in U
\end{align}
are $\mu$-almost everywhere identical. Thus, there holds that $J(u,z) = J(\hat{u},z)$ and hence replacing \eqref{eq:constraint} by \eqref{eq:as_parametric_op} will not affect the solution of the optimization problem \eqref{eq:reduced}.
\end{remark}
}

\subsection{Example}\label{31example}
Let $D \subset \R^d$, $d\in \{1,2,3\}$ denote a bounded domain with Lipschitz boundary $\partial D$ and let $U = [-1,1]^{\N}$ denote \rev{the parameter domain}.
    Let $\alpha \rev{~>~} 0$ and \rev{l}et \rev{the target state} $u_0 \in L^2(D)$. \rev{C}onsider the optimal control problem
\begin{equation}
\min_{z \in L^2(D),u \in L^2_\mu(U,V)} J(u,z)\,,\quad J(u,z) := \frac{1}{2} \int_{U}\, \|u-u_0\|_{L^2(D)}^2\, \rev{\mu(\mathrm d\bsy)} + \frac{\alpha}{2} \| z\|_{L^2(D)}^2\,, \label{eq:1.1_old} 
\end{equation}
subject to the \rev{PDE}
\begin{align}
- \nabla \cdot (a(\pmb y,\pmb x) \nabla u(\pmb y, \pmb x)) &= z(\pmb x) \quad &\pmb x \in D\,,&\quad \pmb y \in U\,,\label{eq:1.2} \\
u(\pmb x) &= 0 \quad &\pmb x \in \partial D\,\rev{.}
\label{eq:1.3}
\end{align}
To ensure wellposedness of \eqref{eq:1.1_old}-\eqref{eq:1.3} we make the following \rev{assumption}: 
\begin{assumption}\label{assump:wellposedness}\rev{
    The uncertainty is described by a sequence $\pmb y = (y_j)_{j\rev{\in \mathbb{N}}}$ of parameters which are independently and identically distributed (i.i.d.) uniformly in $[-1,1]$ for each $j \in \mathbb{N}$, i.e., $\pmb y$ is distributed on $[-1,1]^{\N}$ with probability measure $\mu$, where $\mu(\mathrm d\pmb y) = \bigotimes_{j\rev{\in \mathbb{N}}} \frac{\mathrm dy_j}{2} = \mathrm d\pmb y$. The sequence $\bsy$ enters the diffusion coefficient $a(\pmb y,\pmb x)$ in \eqref{eq:1.2} affinely, that is
\begin{align}\label{eq:KLE}
a(\boldsymbol{y},\boldsymbol{x}):=a_0(\boldsymbol{x})+\sum_{j\in \mathbb{N}}y_j\psi_j(\boldsymbol{x})\,, \quad \boldsymbol{x}\in D\,, \quad \boldsymbol{y}\in  [-1,1]^{\N}\,,
\end{align}
for $a_0 \in L^\infty(D)$, $\psi_j\in L^\infty(D)$ for all $j\in \mathbb{N}$ with $(\|\psi_j\|_{L ^\infty})_{j\in \mathbb{N}}\in\ell^1(\mathbb{N})$, such that
\begin{align*}
0 < a_{\min} \leq a(\pmb y,\pmb x) \leq a_{\max} <\infty\,, \quad \pmb x \in D\,, \quad \pmb y \in  [-1,1]^{\N}\,,
\end{align*}
for some constants $0<a_{\min}\le a_{\max}$ independent of $\pmb x \in D$ and $\bsy \in  [-1,1]^{\N}$.
}
\end{assumption}

The PDE \eqref{eq:1.2} and \eqref{eq:1.3} can be stated in the parametric variational form:
For fixed $\pmb y \in U$, find $u(\bsy) \in V$ such that
\begin{align}
\int_{D} a(\pmb y, \pmb x) \nabla u(\bsy, \pmb x) \cdot \nabla v(\pmb x)\, \mathrm d\pmb x = \int_D z\rev{(\bsx)}\,v\rev{(\bsx)}\, \mathrm d\pmb x \quad \forall v \in H_0^1(D)\,,\label{eq:parametricweakproblem}
\end{align}

By \rev{Assumption~\ref{assump:wellposedness}} the variational form \eqref{eq:parametricweakproblem} is continuous and coercive and by the Lax--Milgram lemma admits a unique solution $u(\bsy) \in H_0^1(D)$ for each $z \in H^{-1}(D)$ and fixed $\bsy \in U$, which satisfies the a-priori bound
\begin{equation*}
    \|u(\bsy)\|_{H_0^1(D)} \leq \frac{\|z\|_{H^{-1}(D)}}{a_{\min}}\,.
\end{equation*}
In particular, the operator $A(\bsy) \in \mathcal{L}(H_0^1(D),H^{-1}(D))$ that can be associated with the bilinear form in \eqref{eq:parametricweakproblem} satisfies $\|A(\bsy)\|_{\mathcal{L}(H_0^1(D),H^{-1}(D))}\leq a_{\max}$ and $\|A(\bsy)^{-1}\|_{\mathcal{L}(H^{-1}(D),H_0^1(D))}\leq 1/a_{\min}$.

For the choices $W = V = H_0^1(D)$, $\calZ = L^2(D)$, $\frakJ = L^2(D)$, the problem \eqref{eq:1.1_old} with \eqref{eq:1.2} and \eqref{eq:1.3} fits into the abstract framework of \eqref{eq:1.1} and \eqref{eq:constraint}. By the general formulation of problem \eqref{eq:1.1} and \eqref{eq:constraint}, we cover a large variety of optimal control problems under uncertainty, including parabolic PDEs, see \cite{Kunoth}.

\subsection{Optimality conditions}
The gradient of $\hat{J}$ as an element of $\calZ$ is given by
\begin{equation*}
    \nabla \hat{J}(z) = \mathcal{B}^{\ast} \int_U (\mathcal{A}^{-1})^{\ast} Q^{\ast} (Q\mathcal{A}^{-1} \mathcal{B}z - u_0)\,\rev{\mu(\mathrm d\bsy)} + \alpha z.
\end{equation*}
The gradient is the unique representer in $\calZ$ of the Fr\'echet derivative $\hat{J}'(z)$ of $\hat{J}$, which belongs to $\calZ'$, i.e., $\hat{J}'(z) = R_{\calZ} \nabla \hat{J}$, where $R_Z: Z \to Z'$ denotes the Riesz operator in the Hilbert space $Z$ given by $\langle z_1, R_{\calZ} z_2\rangle_{\calZ,\calZ'} = \langle z_1, z_2\rangle_{\calZ,\calZ}$. Here and in the following the operator notation $A^*$ for an operator $A:H_1\to H_2$ between two Hilbert $H_1,H_2$ spaces denotes the Hilbert space adjoint, i.e., $\langle Ax, y\rangle_{H_2} = \langle x, A^*y\rangle_{H_1}$.
Defining the adjoint state $q:= (\mathcal{A}^{-1})^{\ast} Q^{\ast} (Q\mathcal{A}^{-1} \mathcal{B}z - u_0)$ leads to the adjoint equation in $L_{\mu}^2(U,V)$
\begin{equation*}
    \mathcal{A}^{\ast}q = Q^{\ast} (Q\mathcal{A}^{-1} \mathcal{B}z - u_0)\,,
\end{equation*}
and the reformulation of the gradient 
\begin{equation}\label{eq:gradient}
    \nabla\hat{J}(z) = \mathcal{B}^{\ast}\int_Uq\,\rev{\mu(\mathrm d\bsy)} + \alpha z\,,
\end{equation}
It is well known (see, e.g., \cite[Theorem 1.46]{hinze2008optimization}) that $z^*  \in \rev{\calZ}$ is a minimizer of  \eqref{eq:reduced} if and only if
\begin{align*}
    \mathcal{A} u &= \mathcal{B}z^*\\
    \mathcal{A}^{\ast} q &= Q^{\ast} (Q\mathcal{A}^{-1} \mathcal{B}z^* - u_0)\\
    \rev{\nabla \hat{J}(z^*)} &\rev{~= 0}.
\end{align*}

\rev{
\begin{remark}
We point out that problem~\eqref{eq:1.1} can be stated with additional control constraints $z \in \calZ_{ad} \subset \calZ$ and state constraints $u \in \calY_{ad} \subset L^2_\mu(U,V)$. Then, the constrained problem has a unique solution, if $\calZ_{ad} \subset \calZ$ and $\calY_{ad} \subset L^2_\mu(U,V)$ are both convex and closed. Moreover, in the case with control constraints the gradient optimality condition $\nabla \hat{J}(z^*)= 0$ is replaced by the variational inequality $\langle \nabla \hat{J}(z^*), z-z^*\rangle_{\calZ} \geq 0$ $\forall z\in \calZ_{ad}$. For optimality conditions with general constraints we refer to~\cite{hinze2008optimization}.
\end{remark}}

\subsection{All-at-once formulation}
\rev{Optimization algorithms for the reduced formulation~\eqref{eq:reduced} typically assume that the \rev{underlying model} can be solved exactly in each iteration, i.e., an existing solver for the solution of the state equation is embedded into an optimization loop.~}
Thereby it is usually preferable to compute the gradient using a sensitivity or adjoint approach, cp.~\eqref{eq:gradient}. \rev{However, the main drawback this approach is that it requires~}
the repeated costly solution of the (possibly nonlinear) state equation, even in the initial stages when the \rev{control} variables are still far from their optimal value. This drawback can be partially overcome by carrying out the early optimization steps with a coarsely discretized PDE and/or only few samples from the space of parameters $U$. 

In this section, we will follow a different approach, which solves the optimization problem and the \rev{underlying model} simultaneously by treating both, the \rev{control} and the state variables, as optimization variables. Various names for the simultaneous solution of the \rev{control} and state equation exist: all-at-once, one-shot method, piggy-back iterations etc. The state and the \rev{control} variables are coupled through the PDE constraint, which is kept explicitly as a side constraint during the optimization.

\subsubsection{Penalty methods}\label{sec:penalty}
A penalty method solves a constrained optimization problem by solving a sequence of unconstrained problems. Using for instance a quadratic penalty method in the present context, one aims to find a sequence of (unique, global) minimizers $(z_k,u_k)$, given by
\begin{equation}
(z_k,u_k) = \argmin_{z_k,u_k} \bigg( J(u_k,z_k) + \frac{\lambda_k}{2} \|e(u_k,z_k)\|_{L^{{2}}_\mu(U,W')}^2 \bigg)\,,
\end{equation}
that converges to the minimizer $(z^\ast,u(z^\ast))$ of the constrained problem \eqref{eq:reduced}. Here, we have defined the residual of \eqref{eq:constraint} as $e(u,z): L^{2}_\mu(U,V) \times \calZ \to L^{2}_{{\mu}}(U,W')$ with $e(u,z) := \mathcal{A}u - \mathcal{B}z$.
The disadvantage of penalty methods is that the penalty parameter $\lambda_k$ needs to be sent to infinity {(for quadratic penalty functions)} which renders the resulting $k$-th problem increasingly ill-conditioned. {The solutions $(z_k,u_k)$ of the sequence of unconstrained optimization problems can be interpreted as maximum a posteriori (MAP) estimates in a Bayesian setting. The quadratic penalty function can then be motivated by an Gaussian model error assumption. From a Bayesian viewpoint, the limit $\lambda_k\to \infty$ corresponds to the small noise limit. This perspective is very appealing, as it allows to introduce and control model error, in particular, more sophisticated choices such as distributed penalty function can be straightforwardly incorporated. We refer to \cite{GSW2020} for more details  on the connection between the penalty and Bayesian approach.}

\section{Surrogates}\label{sec:surrogates}
\rev{In the field of uncertainty quantification mathematical models governed by PDEs with uncertain coefficients are often parameterized in terms of a countably infinite number of parameters $\bsy = (y_j)_{j\in \mathbb{N}} \in U$ and require a large finite number of parameters for accurate approximation. Thus, throughout the rest of this paper we assume that the parameter domain $U$ is of the form $U = \times_{j\in \mathbb{N}}~U_j$, where $U_j$ is assumed to be a compact Hausdorff space for each $j\in \mathbb{N}$. Since parameterized models are computationally expensive to evaluate,
r}eplacing the solution of the forward model by a surrogate, that is cheap to evaluate, can be a tremendous advantage. In the next sections we will analyze the optimization problem in which the parametric mapping is replaced by a surrogate, i.e., the mapping $u\rev{(\bsy)}: U \to V$ is replaced by a surrogate
\begin{align*}
u(\theta,\bsy): \Theta \times U \to V
\end{align*}
where $\theta \in \Theta$ are the parameters of the surrogate. To do so we introduce the following notation: a multiindex is denoted by $\bsnu = (\nu_j)_{j\in \mathbb{N}} \in \mathbb{N}_0^{\mathbb{N}}$.
We denote the set of all indices with finite support by $\mathcal{F} := \{\bsnu \in \mathbb{N}_0^{\mathbb{N}}: \sum_{j\in \mathbb{N}} \nu_j < \infty\}$ and follow the conventions $\bsnu! := \prod_{j \in \mathbb{N}} \nu_j$ and $0! := 1$. Moreover, for a sequence $\bsy = (y_j)_{j\in \mathbb{N}}$ and a multiindex $\bsnu \in \mathcal{F}$ we write $\bsy^\bsnu := \prod_{j\in\mathbb{N}} y_j^{\nu_j}$, using the convention $0^0 := 1$.
By $|\Lambda|$ we denote the finite cardinality of a set $\Lambda$.
For a real Banach space $V$, its complexification is the space $V_{\mathbb{C}} := V + iV$ with the Taylor norm $\|v+iw\|_{V_{\mathbb{C}}} := \sup_{t\in [0,2\pi)}\|\cos(t)v - \sin(t)w\|_V$ for all $v,w \in V$ and $i$ denoting the imaginary unit. \rev{Finally, for $r \in (0,\infty)$ we denote $B_r := \{z \in \mathbb{C} : |z|<r\}$.}

Possible surrogates include for instance
\begin{itemize}
\item a power series of the form 
\begin{align}\label{eq:power}
u(\theta,\bsy) = \sum_{\boldsymbol{\nu}\in \mathcal{F}} \theta_{\boldsymbol{\nu}} \bsy^{\boldsymbol{\nu}},
\end{align}
\rev{where $\theta = (\theta_{\boldsymbol{\nu}})_{\boldsymbol{\nu}\in \calF}$.}
\item an orthogonal series of the form
\begin{align}\label{eq:Legendre}
u(\theta,\bsy) = \sum_{\boldsymbol{\nu}\in \mathcal{F}} \theta_{\boldsymbol{\nu}} P_{\boldsymbol{\nu}}\,, \quad P_{\boldsymbol{\nu}}:= \prod_{j\rev{\in \mathbb{N}}} P_{\nu_j}(y_j)\,,
\end{align}
\rev{where $\theta = (\theta_{\boldsymbol{\nu}})_{\boldsymbol{\nu}\in \calF}$, and }$P_k$ is the Legendre polynomial of degree $k$ defined on $[-1,1]$ and normalized with respect to the uniform measure, i.e., such that $\int_{-1}^1|P_k(t)|^2 \frac{\mathrm dt}{2} = 1$. 
\item a neural network 
with $L \in \mathbb{N}$ layers, defined by the recursion
\begin{align}
x_{\ell} &:= \sigma(W_{\ell}x_{\ell-1} + b_{\ell})\,, \quad \text{for } \ell = 1,\ldots,L-1\,,\notag\\
u(\theta,\bsy) &:= W_{L}x_{L-1} + b_{L},\quad x_0 := \bsy\,.\label{eq:neuralnetwork}
\end{align}
Here the parameters $\theta \in \Theta := \times_{\ell=1}^L (\mathbb{R}^{N_\ell \times N_{\ell-1}} \times \mathbb{R}^{N_\ell})$ are a sequence of matrix-vector tuples
\begin{align*}
\theta = \big((W_{\ell},b_{\ell})\big)_{\ell = 1}^L = \big(W_1,b_1),(W_2,b_2),\ldots,(W_L,b_L)\big)\,,
\end{align*}
and the activation function $\sigma$ is applied component-wise to vector-valued inputs.
\item Gaussian process or kernel based approximations. Recently, a general framework for the approximation of solution of nonlinear~\rev{PDE}s has been proposed in \cite{chen2021solving}. The authors demonstrate the efficiency of Gaussian processes for nonlinear problems and derive a rigorous convergence analysis. We refer to \cite{chen2021solving} for more details, in particular also to the references therein. 
\item reduced basis or low rank approaches, which have been demonstrated to efficiently approximate the solution of the \rev{underlying} problem even in high or \rev{infinite} dimensional settings, see e.g., \cite{10.1093/imanum/drx052,rozza:hal-01722593}.
\end{itemize} 
There has been a lot of research towards efficient surrogates, in particular in the case of parametric PDEs and the above list is by far not exhaustive. We provide in the following a general framework to train surrogates simultaneously with the optimization step and illustrate the ansatz for polynomial chaos and neural network approximations.

Based on the smoothness of the underlying function, approximation results of the above surrogates can be stated.~\rev{In the remainder of this section, we will delve further into the smoothness of the underlying function and briefly discuss the approximation results of the above surrogates.} To this end, we recall that
the solution $u(\bsy)$ of a parametric linear operator equation \eqref{parametric_op} is an analytic function \rev{with respect }to the parameters $\bsy$, if the linear operators $A(\bsy) \in \mathcal{L}(V,W')$ are isomorphisms and as long as the operator $A$ and the right-hand side $z$ are parametrized in an analytical way, see e.g., \cite[Theorem 1.2.37]{Zech2018}, or \cite[Theorem 4]{Kunoth} which in addition provides bounds on the partial derivatives \rev{with respect }to the parameters. Moreover, analytic functions between Banach spaces admit holomorphic extensions, i.e., for analytic $f: U \to Y$ between two real Banach spaces $X$ and $Y$ with $U\subseteq X$ open, there exists an open set $\tilde{U}\subseteq X_{\mathbb{C}}$ and a holomorphic extension $\widetilde{f}:\widetilde{U} \to Y_{\mathbb{C}}$ such that $U\subseteq \widetilde{U}$ and $\widetilde{f}|_{U} = f$, see \cite[Proposition 1.2.33]{Zech2018}.
To quantify the smoothness of the underlying function we will use the notion of $(\bsb,\epsilon)$-holomorphy of a function, which is a sufficient criterion for many approximation results, see \cite{SZ19} and the \rev{references} therein:
Given a monotonically decreasing sequence $\boldsymbol{b} = (b_j)_{j\in\mathbb{N}}$ of positive real numbers that satisfies $\boldsymbol{b} \in \ell^p(\mathbb{N})$ for some $p \in (0,1]$, a continuous function $u\rev{(\bsy}): U \to V$ is called $(\bsb,\epsilon)$-holomorphic if
for any sequence $\bsrho := (\rho_j)_{j\rev{\in \mathbb{N}}} \in [1,\infty)^{\mathbb N}$, satisfying
\begin{align*}
\sum_{j \rev{\in \mathbb{N}}} (\rho_j-1) b_j \leq \epsilon\,,
\end{align*}
for some $\epsilon >0$, there exists a complex extension $\widetilde{u}: B_{\bsrho} \to V_{\mathbb{C}}$ of $u$, where $B_{\bsrho} := \times_{j\in \mathbb{N}} B_{\rho_j}$, with $\widetilde{u}(\bsy) = u(\bsy)$ for all $\bsy \in U$, such that $w \mapsto \widetilde{u}(\boldsymbol{w}): B_{\bsrho} \to V_{\mathbb{C}}$ is holomorphic as a function in each variable $w_j \in B_{\rho_j}$, $j \in \mathbb{N}$ with uniform bound
\begin{align*}
\sup_{\boldsymbol{w} \in B_{\bsrho}} \|u(\boldsymbol{w})\|_{V_{\mathbb{C}}} \leq C\,.
\end{align*}

The sequence $\bsb$ determines the size of the domains of the holomorphic extension, i.e., the faster the decay in $\bsb$, the faster the radii $\rho_j$ may increase. Furthermore, the summability exponent $p$ of the sequence $\bsb \in \ell^p(\N)$ will determine the algebraic convergence rates of the surrogates.~\rev{For details on the approximation of $(\bsb,\epsilon)$-holomorphic functions we refer to~\cite{CohenDeVore15} for generalized polynomial chaos expansions, and to~\cite{SZ19} for the approximation by neural networks.}

\begin{remark}\rev{ (i) Setting $\bsb := (\|\psi_j\|_{L^\infty(D)})_{j\in \mathbb{N}}$ and assuming in addition to \rev{Assumption~\ref{assump:wellposedness}} that $\bsb \in \ell^p(\mathbb{N})$ for some $p\in (0,1)$, the parametric solution $u\rev(\bsy)$ of the uniformly elliptic problem \eqref{eq:parametricweakproblem} is $(\bsb,\epsilon)$-holomorphic, see e.g., \cite{CohenDeVore15, SZ19}. Thus, the polynomial expansions and the neural network can be applied to approximate the elliptic PDE problem in Subsection \ref{31example}.}

\rev{(ii) We note also that the series expansions \eqref{eq:power} and \eqref{eq:Legendre} are linear in the parameters $\theta$, whereas the neural network is nonlinear in its parameters due to the nonlinear activation function $\sigma$.}
\end{remark}

\section{Consistency analysis}\label{sec:consistency}

In this section we first analyze the truncation error in Subsection~\ref{sec:dimtrunc} before we study consistency \rev{with respect }to the number of data points and the penalty weight in Subsection~\ref{sec:consistency2}.

\subsection{Dimension truncation}
\label{sec:dimtrunc}
Let $z^*$ be the solution of problem \eqref{eq:reduced} with infinite-dimensional parameter $\bsy = (y_j)_{j\rev{\in \mathbb{N}}}$ and let $z_s^*$ be the solution of the truncated problem, i.e., the solution of \eqref{eq:reduced} with $\bsy$ having only $s$-many non-zero components $(\bsy_{\leq s},\boldsymbol{0}) = (y_1,\ldots,y_s,0,0,\ldots)$ and involving integrals over the $s$-dimensional domain \rev{$U_s$, that is}
\begin{align*}
    \rev{
    \min_{z \in \rev{\mathcal{Z}}} \hat J_s(z)\,,\quad \hat J_s(z) := \frac{1}{2} \int_{U_s} \|\mathcal{Q}\mathcal{A}^{-1}\mathcal{B}z-u_0\|^2_{\frakJ}\,\rev{\mu(\mathrm d(\bsy_{\le s},\boldsymbol{0}))} + \frac{\alpha}{2} \|z\|^2_{\calZ}\,.
}\end{align*}
Then, there holds that $\langle {\nabla}\hat{J}_s(z_s^*), z^*-z^*_s\rangle_{\calZ} \geq 0$ and $\langle -{\nabla} \hat{J}(z^*), z^*-z^*_s\rangle_{\calZ} \geq 0${, where $\nabla \hat{J}_s$ is the gradient of the truncated objective functional}. Hence it follows by adding the two inequalities that 
  $\langle {\nabla}\hat{J}_s(z^*_s) - {\nabla}\hat{J}(z^*),z^*-z^*_s\rangle_{\calZ}\ge0.$
Let $\alpha>0$, then with the $\alpha$-strong convexity of $\hat{J}_{\rev{s}}$, i.e., \rev{$\langle \nabla \hat{J}_s(z_s^*) - \nabla \hat{J}_s(z^*) + \alpha (z^* - z^*_s), z^* - z^*_s \rangle_{\calZ} \le 0$}, and the Cauchy--Schwarz inequality one obtains
\begin{align*}
    \alpha \|z^* - z^*_s\|_{\calZ}^2 &\leq \alpha \|z^* - z^*_s\|_{\calZ}^2+ \langle {\nabla}\hat{J}_{\rev{s}}(z^*_s) - {\nabla}\hat{J}(z^*),z^*-z^*_s\rangle_{\calZ} \\
    &\rev{~= \langle \nabla \hat{J}_s(z_s^*) - \alpha z_s^* - \nabla \hat{J}(z^*) + \alpha z^*, z^* - z_s^* \rangle_{\calZ} }\\
    &\rev{~=\langle \nabla \hat{J}_s(z_s^*) - \alpha z^*_s - \nabla \hat{J}_s(z^*) + \alpha z^*, z^* - z^*_s \rangle_{\calZ}}\\
    &\quad \rev{~+ \langle \nabla \hat{J}_s(z^*) - \alpha z^* - \nabla \hat{J}(z^*) + \alpha z^*, z^* - z^*_s \rangle_{\calZ}}\\
    & \leq \|{\nabla}\hat{J}(z^*) - {\nabla}\hat{J}_s(z^*)\|_{\calZ} \|z^* - z^*_s\|_{\calZ}\,,
\end{align*}
and finally
\begin{equation}\label{eq:trunc_error}
    \|z^* - z^*_s\|_{\calZ} \leq \alpha^{-1} \|\mathcal{B}^* \int_U q(\bsy,z^*) - q((\bsy_{\leq s},\boldsymbol{0}),\rev{z^*})\,\rev{\mu(\mathrm d\bsy)}\|_{\calZ}\,.
\end{equation}
In order to control the dimension truncation error of the optimal controls it is thus sufficient to bound the dimension truncation error of the adjoint state. To derive a rigorous error bound and convergence rate, we specify the regularity assumptions of the forward operator, see also \cite[Assumption 1]{Kunoth}.
\begin{assumption}\label{assump:analytic_op}
The parametric operator family $\{A(\bsy) \in \mathcal{L}(V,W') : \bsy \in U\}$ is a regular $p$-analytic operator family for some $0 < p \leq 1$, that is
\begin{itemize}
	\item[(i)] The operator $A(\bsy)$ is invertible for every $\bsy \in U$ with uniformly bounded inverse $A^{-1}(\bsy) \in \mathcal{L}(W',V)$, i.e., there exists $C>0$ such that 
$\sup_{\bsy \in U} \|A(\bsy)^{-1}\|_{\mathcal{L}(W',V)} \leq C\,.$
	\item[(ii)] For each $\bsy \in U$, the operator $A(\bsy)$ is a real analytic function with respect to $\bsy$. Precisely, this means there exists a nonnegative sequence $\bsb = (b_j)_{j\in \mathbb{N}} \in \ell^p(\mathbb{N})$ such that for all $\bsnu \in \calF\setminus\{0\}$ and the same $C$ as in (i) it holds that
\begin{align}\label{eq:bsb_Assumpt}
\sup_{\bsy \in U} \|A(\boldsymbol{0})^{-1} \partial_{\bsy}^{\bsnu} A(\bsy)\|_{\mathcal{L}(V,V)} \leq C\bsb^{\bsnu}\,.
\end{align}
\end{itemize}
\end{assumption}
In \cite[Theorem 4]{Kunoth}
the authors show that for a parametric family of operators $\{A(\bsy)\in \mathcal{L}(V,W'): \bsy \in U\}$ satisfying Assumption \ref{assump:analytic_op} for some $0<p\leq 1$, for $\rev{z}\in W'$, and every $\bsy \in U$ there exists a unique solution $u(\bsy) \in V$ of the parametric operator equation %
\begin{align}\label{eq:operator_eq}
A(\bsy) u(\bsy) = \rev{z}
\end{align}
and the parametric solution family $u(\bsy)$ depends analytically on the parameters $\bsy \in U$, with partial derivatives satisfying
\begin{align*}
\sup_{\bsy \in U} \|\partial_{\bsy}^{\bsnu} u (\bsy)\|_V \leq C \|\rev{z}\|_{W'} |\bsnu|! \frac{\bsb^{\bsnu}}{(\ln{2})^{|\bsnu|}}\,,
\end{align*}
where $\bsb$ is defined in \eqref{eq:bsb_Assumpt}.
Compared to \eqref{eq:operator_eq} the right-hand side of the adjoint state equation depends in addition on the parameters $\bsy \in U$. In particular, in \cite[Corollary 3]{Kunoth} it is shown that affine parametric operators, as in~Example \ref{31example}, are $p$-analytic. We next show that the adjoint state depends analytically on the parameters with similar bounds on the partial derivatives.

\begin{theorem}
Let the parametric family of operators $\{A(\bsy) \in \mathcal{L}(V,W'): \bsy \in U\}$ satisfy Assumption \ref{assump:analytic_op} for some $0<p\leq 1$. Then, for $f(\bsy) := \mathcal{Q}^*(\mathcal{Q}{A(\bsy)}^{-1}\mathcal{B}z - u_0) \in V$ and every $\bsy \in U$ there exists a unique solution $q(\bsy) \in W'$ of the parametric operator equation 
\begin{equation}\label{eq:adjoint_para}
    ({A}(\bsy))^*q(\bsy) = f(\bsy)\,.
\end{equation}
Moreover, if $\|\partial_{\bsy}^{\bsnu} f(\bsy)\|_{V} \leq C (\|z\|_{\calZ} + \|u_0\|_{\frakJ}) |\bsnu|! \bsb^{\bsnu}/(\ln{2}^{|\bsnu|})$ for all finitely supported multiindices $\bsnu \in \calF$, then the parametric solution family $q(\bsy)$ depends analytically on the parameters $\bsy \in U$, with partial derivatives satisfying
\begin{equation*}
\sup_{\bsy \in U} \|\partial_{\bsy}^{\bsnu} q (\bsy)\|_{W'} \leq C (\|z\|_{\calZ} + \|u_0\|_{\frakJ}) (|\bsnu|+1)! \frac{\bsb^{\bsnu}}{(\ln{2})^{|\bsnu|}}\,.
\end{equation*}
\end{theorem}
\begin{proof}
We prove the result by induction with respect to $|\bsnu|$. If $|\bsnu| = 0$, then $\bsnu = \boldsymbol{0}$ and the result follows from Assumption \ref{assump:analytic_op} (i) and the boundedness of $A,\mathcal{B}$ and $\mathcal{Q}$. For $\boldsymbol{0} \neq \bsnu \in \calF$ we take the partial derivative $\partial_{\bsy}^{\bsnu}$ of \rev{\eqref{eq:adjoint_para}}. By the Leibniz product rule we get
\begin{align*}
\sum_{\bsm \leq \bsnu} \binom{\bsnu}{\bsm} (\partial_{\bsy}^{\bsnu-\bsm} A(\bsy)^* (\partial_{\bsy}^{\bsm} q(\bsy)) = \partial_{\bsy}^{\bsnu} f(\bsy)\,.
\end{align*}
Separating out the $\bsnu = \bsm$ term, we obtain
\begin{align*}
A(\bsy)^*(\partial_{\bsy}^{\bsnu} q(\bsy)) = -\sum_{\bsm \leq \bsnu, \bsm \neq \bsnu} \binom{\bsnu}{\bsm} (\partial_{\bsy}^{\bsnu-\bsm} A(\bsy)^*) (\partial_{\bsy}^{\bsm} q(\bsy)) + \partial_{\bsy}^{\bsnu}f(\bsy)\,.
\end{align*}
Note that there holds
\begin{align*}
    \|A(\bsy)\|_{\mathcal{L}(V,W')} = \|A(\bsy)^*\|_{\mathcal{L}(W',V)}
\end{align*}
and for all $\bsnu \in \calF$ that
\begin{align}\label{eq:operator_deriv}
    \|\partial^\bsnu_\bsy A(\bsy)\|_{\mathcal{L}(V,W')} = \|\partial^\bsnu_\bsy A(\bsy)^*\|_{\mathcal{L}(W,V')}\,.
\end{align}
This holds true since
\begin{align*}
    \|\partial^{\bsnu}_{\bsy} A(\bsy)\|_{\mathcal{L}(V,W')} 
    &=\sup_{v\in V} \sup_{w \in W'} \frac{|\langle \partial^{\bsnu}_{\bsy} A(\bsy) v, w\rangle_{W'}|}{\|v\|_V \|w\|_{W'}}
    =\sup_{v\in V} \sup_{w \in W'} \frac{|\partial^{\bsnu}_{\bsy} \langle A(\bsy) v, w\rangle_{W'}|}{\|v\|_V \|w\|_{W'}}\\ 
    &=\sup_{v\in V} \sup_{w \in W'} \frac{|\partial^{\bsnu}_{\bsy} \langle v, A(\bsy)^* w\rangle_{V}|}{\|v\|_V \|w\|_{W'}}
    =\sup_{v\in V} \sup_{w \in W'} \frac{|\langle v, \partial^{\bsnu}_{\bsy}A(\bsy)^* w\rangle_{V}|}{\|v\|_V \|w\|_{W'}} \\
    &= \sup_{w\in W'} \frac{\|\partial^{\bsnu}_{\bsy}A(\bsy)^*w\|_{V}}{\|w\|_{W'}}
    = \|\partial^{\bsnu}_{\bsy}A(\bsy)'\|_{\mathcal{L}(W',V)}\,.
\end{align*}
Thus, we obtain
\begin{align*}
\|\partial_{\bsy}^{\bsnu} q(\bsy)\|_{W'} 
&\leq \sum_{\bsm \leq \bsnu, \bsm \neq \bsnu} \binom{\bsnu}{\bsm} \|(A(\bsy)^*)^{-1} \partial_{\bsy}^{\bsnu-\bsm}A(\bsy)^*\|_{\mathcal{L}(W')} \|\partial_{\bsy}^{\bsm} q(\bsy)\|_{W'}\\ 
&\quad+ \|(A(\bsy)^*)^{-1}\|_{\mathcal{L}(V,W')} \|\partial_{\bsy}^{\bsnu} f(\bsy)\|_{V}\\
&\leq \sum_{\bsm \leq \bsnu, \bsm \neq \bsnu} \binom{\bsnu}{\bsm} C \bsb^{\bsnu-\bsm} \|\partial_{\bsy}^{\bsm} q(\bsy)\|_{W'} + C\|\partial_{\bsy}^{\bsnu} f(\bsy)\|_{V}\,.
\end{align*}
where we concluded from Assumption \ref{assump:analytic_op} and \eqref{eq:operator_deriv} that for all $\bsnu \in \calF$
\begin{align*}
\sup_{\bsy \in U}\|(A(\bsy)^*)^{-1} \partial_{\bsy}^{\bsnu} A(\bsy)^*\|_{\mathcal{L}(W')} &\leq \sup_{\bsy \in U} \|(A(\bsy)^*)^{-1} (A(\boldsymbol{0}))^*\|_{\mathcal{L}(W')} \sup_{\bsy \in U} \|(A(\boldsymbol{0})^*)^{-1}\partial_{\bsy}^{\bsnu} A(\bsy)^*\|_{\mathcal{L}(W')} \\
&\leq C \bsb^{\bsnu}\,.
\end{align*}
From Lemma \cite[Lemma 5]{kuo2017multilevel} we conclude that
\begin{align*}
\|\partial_{\bsy}^{\bsnu} q(\bsy)\|_{W'}  \leq C \sum_{\bsk \leq \bsnu} \binom{\bsnu}{\bsk} \frac{|\bsk|!}{(\ln{2})^{|\bsk|}} \bsb^{\bsk} \|\partial_{\bsy}^{\bsnu-\bsk} f(\bsy)\|_{V}\,.
\end{align*}
By the assumption on $f$ we obtain
\begin{align*}
\|\partial_{\bsy}^{\bsnu} q(\bsy)\|_{W'}
&\leq C \sum_{\bsk \leq \bsnu} \binom{\bsnu}{\bsk} \frac{|\bsk|!}{(\ln{2})^{|\bsk|}} \bsb^{\bsk} C (\|z\|_{\calZ} + \|u_0\|_{\frakJ}) \frac{\bsb^{\bsnu-\bsk}}{(\ln{2})^{|\bsnu-\bsk|}} |\bsnu-\bsk|!\\
&= C (\|z\|_{\calZ} + \|u_0\|_{\frakJ}) (|\bsnu| + 1)! \frac{\bsb^{\bsnu} }{(\ln{2})^{|\bsnu|}}\,,
\end{align*}
where $C>0$ is a generic constant that might change within one equation.
\end{proof}
Since $\mathcal{B}$ is a bounded operator we can apply \cite[Theorem 6.2]{GKKSS22} to establish a convergence rate for the error in \eqref{eq:trunc_error}.
\begin{theorem}[{\cite[Theorem 6.2]{GKKSS22}}]\label{thm:dim_trunc}
Let Assumption~\ref{assump:analytic_op} hold, and let $q(\bsy)$ be the solution of \eqref{eq:adjoint_para} for a control $z\in \calZ$ and a given target state $u_0 \in \frakJ$. For a sequence of uniformly distributed parameters $\bsy \in U =
\rev{[-1,1]^{\mathbb{N}}}$, denoting $(\bsy_{\leq s};\boldsymbol{0}) =
(y_1,y_2,\ldots,y_s,0,0,\ldots)$, we have for all $s \in \mathbb{N}$
\begin{align*}
\Big\| \int_U \big(q(\bsy) - q(\bsy_{\leq s};\boldsymbol{0})\big)\,\mathrm d\bsy \Big\|_{W'}
 \le C\,(\|z\|_{\calZ} +\|u_0\|_{\mathfrak{J}})s^{-2/p+1}\,,
\end{align*}
for $C>0$ independent of $s$ and $p$ as in Assumption~\ref{assump:analytic_op}.
\end{theorem}
Dimension truncation error rates for Bochner integral quantities are derived in \cite{GK22} for a more general class of distributions.
Since the truncation error can be controlled, this error contribution is not considered in the remaining analysis of this manuscript.

\subsection{Penalty and quadrature error}
\label{sec:consistency2}

In our consistency analysis, we are going to analyse the proposed penalty method, see Section \ref{sec:penalty}, \rev{with respect }to the penalty parameter $\lambda_k$ and the number of i.i.d.~data points $N$, denoted as $(\bsy^i)_{i=1}^N$, which are used to approximate the expected values \rev{with respect }to $\bsy$. Thereby, we assume that the state $u_k\rev{(\bsy)}$ has been parametrized by a surrogate $u(\theta,\bsy)$, see Section \ref{sec:surrogates}, and the penalty parameters $(\lambda_k)_{k\in\mathbb N}$ are monotonically increasing to infinity. In particular, we~\rev{shall}~connect the following optimization problems: 
\begin{enumerate}
\item The original constrained risk minimization (cRM) problem 
\begin{equation*}
\begin{split}
 \min_{z,\theta}\ \frac12\E_\bsy [\|\calQ u(\theta,\bsy)-u_0\|_{\frakJ}^2] + \frac\alpha2 \|z\|_{\calZ}^2&\\
 \text{subjected to}\quad \E_\bsy[\|e(u(\theta,\bsy),z)\|_{W'}^2] = 0.&
 \end{split}
\end{equation*} 
We assume there exists a {unique} solution of this problem, which we will denote by $(z_\infty^\ast, \theta_\infty^\ast)$. 
\item The penalized risk minimization (pRM) problem 
\begin{equation}\label{eq:pRM}
 \min_{z,\theta}\  \frac12\E_\bsy [\|\calQ u(\theta,\bsy)-u_0\|_{\frakJ}^2] + \frac\alpha2 \|z\|_{\calZ}^2 + \frac{\lambda_k}2\E_\bsy [\|e(u(\theta,\bsy),z)\|_{W'}^2].
\end{equation} 
We assume there exists a unique solution denoted by $(z_\infty^k,\theta_\infty^k)$.

\item The penalized empirical risk minimization (pERM) problem
\begin{equation}\label{eq:pERM}
 \min_{z,\theta}\  \frac{1}{2N} \sum_{i=1}^N \|\calQ u(\theta,\bsy^i)-u_0\|_{\frakJ}^2 + \frac\alpha2 \|z\|_{\calZ}^2 + \frac{\lambda_k}2 \frac1N\sum_{i=1}^N \|e(u(\theta,\bsy^i),z)\|_{W'}^2
\end{equation} 
We {assume there exists a unique solution} denoted by $(z_N^k,\theta_N^k)$.
\end{enumerate}

For simplicity, in the following we denote $x=(z,\theta)\in \calX := \calZ\times\R^d$ and define the functions
\begin{align*}
f:\calX \times U \to \R_+,&\quad \text{with}\quad f(x,\bsy) = \frac12\|\calQ u(\theta,\bsy) - u_0\|_{\frakJ}^2+\frac\alpha2\|z\|_{\calZ}^2,\\
g: \calX \times U \to \R_+,&\quad \text{with}\quad g(x,\bsy) = \|e(u(\theta,\bsy),z)\|_{W'}^2,
\end{align*}

where we assume here and in the following that $W'$ and hence also $L^2_{\mu}(U,W')$ are Hilbert spaces.

\subsection{Convergence of pRM to cRM}
We start with the error dependence on the penalty parameter $\lambda_k$. 
The following is a long known result (see e.g., \cite[Theorem 1]{POLYAK19711}) providing unique existence of solutions as well as convergence towards the unconstrained problem for increasing penalty parameter. 
\begin{theorem}\label{thm:penalty_conv}
Let $H_1$ and $H_2$ be two Hilbert spaces and let $\rev{\cJ}(x)$ be a functional on $H_1$ and the constraint $h(x)$ be an operator from $H_1$ into $H_2$. Moreover, suppose 
\begin{itemize}
\item there exists a unique global minimizer $x^\ast \in \calX$ of the problem
\begin{align*}
\min_{x \in \calX} \rev{\cJ}(x) \quad \text{s.t. } h(x) = 0 \text{ in } H_2\,.
\end{align*}
\item that $\nabla_x \rev{\cJ}(x),\nabla_x^2 \rev{\cJ}(x)$ and $\nabla_xh(x),\nabla_x^2h(x)$ exist with 
\begin{align*}
\|\nabla_x^2 \rev{\cJ}(x) - \nabla_x^2 \rev{\cJ}(y)\|_{\mathcal{L}(H_1,\mathcal{L}(H_1,\mathbb{R}))} &\leq L_1 \|x-y\|_{H_1} \\ \quad \text{and} \quad \|\nabla_x^2h(x) - \nabla_x^2h(y)\|_{\mathcal{L}(H_1,\mathcal{L}(H_1,H_2))} &\leq L_2 \|x-y\|_{H_1}\,.
\end{align*}
\item the linear operator $\nabla_xh(x^\ast)$ is non-degenerate, i.e., $\|(\nabla_xh(x^\ast))^\ast y\|_{H_2} \geq c\|y\|_{H_2}$ for $c>0$ and for all $y \in H_2$.
\item the self-adjoint operator $\nabla_x^2L(x^\ast,y^\ast)$ is positive definite, i.e., $\langle \nabla_x^2L(x^\ast,y^\ast) \tilde{x}, \tilde{x}\rangle \geq m \|\tilde{x}\|_{H_1}^2$ for $m>0$ and all $\tilde{x} \in H_1$. Here, the functional $L$ denotes the Lagrangian and the Lagrange multiplier rule is applicable because of the first three assumptions\rev{, i.e., there exists $y^* \in H_2$ such that $L_x(x^*,y^*) = \nabla_x \mathcal{J}(x^*) + (\nabla_x h(x^*))^* y^* = 0$.}
\end{itemize}
Then, for sufficiently large $\lambda_k>0$, there exists a unique minimizer $x_k^\ast$ of the problem
\begin{align*}
\min_{x \in H_1} \rev{\cJ}(x) + \frac{\lambda_k}{2} \|h(x)\|_{H_2}^2
\end{align*}
which satisfies
\begin{align*}
\|x_k^\ast - x^\ast\|_{H_1} \leq \frac{C}{2\lambda_k}\|y^\ast\|_{H_2}
\quad\text{ and }\quad
\|\lambda_k h(x_k^\ast) - y^\ast\|_{H_1} \leq \frac{C}{2\lambda_k}\|y^\ast\|_{H_2}\,.
\end{align*}
\end{theorem} 
This theorem holds in infinite-dimensional Hilbert spaces $H_1$ and $H_2$, in that case the derivatives \rev{with respect }to $x\in H_1$ in the theorem are Fr\'echet derivatives. For 
{our considered problem at hand with $x = (z,\theta)\in \calX = \calZ\times \R^d$} the assumptions need to be satisfied for $\rev{\cJ}(x) = \frac12\mathbb{E}_{\bsy}[\|\calQ u(\theta,\bsy) - u_0\|_{\frakJ}^2]+\frac\alpha2\|z\|_{\calZ}^2$, $h(x) = e(u(\theta,\bsy,z))$ and $g(x)=\|h(x)\|_{W'}^2$ based on the spaces $H_1 = \calX$ and $H_2 = L^2_\mu(U,W')$. Here we need the assumption that $W'$ is a Hilbert space, such that $L^2_\mu(U,W')$ is a Hilbert space. In this case the $H_1$-norm is the norm on $\mathcal X$, e.g., $\|(z,\theta)\|_{\mathcal X} = (\|z\|_{\calZ}^2+\|\theta\|_2^2)^{1/2}$, 
and the $H_2$-norm is $\|\cdot\|_{L^2_\mu(U,W')} := (\mathbb{E}_{\bsy}[\|\cdot\|_{W'}^2])^{1/2}$. 

If a surrogate satisfies the assumptions of the preceding theorem, the convergence of the minimizers of the pRM problem to the minimizer of the cRM problem is guaranteed.

\begin{lemma}\label{lem:pRM_CRM}
Suppose that $f$ and $g$ satisfy the  assumptions of Theorem~\ref{thm:penalty_conv}. Then the solution of the pRM problem converges to the solution of the cRM problem, in the sense that there exists $C_1>0$ independent of $N$ such that
$\|(z_\infty^k,\theta_\infty^k) - (z_\infty^\ast,\theta_\infty^\ast)\|_{{\mathcal X}}^2 \le \frac{C_1}{\lambda_k^2}.$
\end{lemma} 

\subsection{Convergence of pERM  to  pRM}

The following result describes the error arising due to the empirical approximation of the risk function. 

\begin{lemma}\label{lem:pERM_pRM}
Suppose that $f$ is strongly convex in the sense that $\nabla_x^2 f(x,\bsy)~\rev{\ge}~m\cdot{\mathrm{id}}$ for all $x\in \calX =\calZ\times \R^d$ and $\bsy\in U$ and 
$g$ is convex.  Let $\lambda_k>0$ and assume that 
\rev{$\Tr(\Cov_\bsy(\nabla_x (f(x_\infty^k,\bsy)+\frac{\lambda_1}2 g(x_\infty^k,\bsy))))<\infty$}.
Then the solution of \eqref{eq:pERM} converges 
to the solution $x_\infty^k=(z_\infty^k,\theta_\infty^k)$ of \eqref{eq:pRM}, in the sense that there exists $C_{\lambda_k}>0$ dependent of $\lambda_k$ 
such that 
\begin{equation*}
\E_\bsy[\|(z_N^k,\theta_N^k) - (z_\infty^k,\theta_\infty^k)\|_{{\mathcal X}}^2] \le \frac{C_{\lambda_k}}{N}.
\end{equation*}
\end{lemma} 
\begin{remark}
    \rev{Note that $\Cov_\bsy$ denotes the multivariate covariance operator with respect to the random variable $\bsy$. This is, for two measurable functions $h_1:U\to \mathcal X$, $h_2:U\to\mathcal X$, the covariance operator is defined by
    \[ \Cov_\bsy (h_1(\bsy),h_2(\bsy)) := \E_\bsy\left[(h_1(\bsy)-\E_\bsy[h_1(\bsy)])\otimes (h_2(\bsy)-\E_\bsy[h_2(\bsy)])\right]\,,\]
    and for simplicity we define
    \[ \Cov_\bsy (h_1(\bsy)):=\Cov_\bsy (h_1(\bsy),h_1(\bsy))\,.\]
    In the following, we will make use of the property
    \begin{align*}
        \Tr(\Cov_\bsy(h(\bsy))) &= \E_\bsy\left[\Tr\left((h(\bsy)-\E_\bsy[h(\bsy)])\otimes(h(\bsy)-\E_\bsy[h(\bsy)])\right)\right]\\ &= \E_\bsy\left[\|h(\bsy)-\E_\bsy[h(\bsy)]\|_{\mathcal X}^2 \right]\,,
    \end{align*}
    where we have used that the trace operator is linear. 
    }
\end{remark}
\begin{proof}
Under the above assumption the objective function in \eqref{eq:pERM} is strongly convex. The unique solution $x_N^k$ satisfies \[ \frac1N\sum_{i=1}^N \nabla_x f(x_N^k,\bsy^i) + \frac{\lambda_k}2 \frac1N\sum_{i=1}^N \nabla_x g(x_N^k,\bsy^i) 
= 0.\]

Similarly, the unique \rev{minimizer} of \eqref{eq:pRM} is characterized by
\[ \nabla_x\E_\bsy[f(x_\infty^k,\bsy)] + \frac{\lambda_k}2 \nabla_x\E_\bsy[g(x_\infty^k,\bsy)] 
= 0.\]
We are now interested in the discrepancy of $x_N^k$ and $x_\infty^k$.  We define the (empirical) risk functions 
\begin{align*} 
\Psi^k(x) &= \E_\bsy[f(x,\bsy)] + \frac{\lambda_k}2  \E_\bsy[g(x,\bsy)] 
,\quad
\Psi_N^k(x) = \frac1N\sum_{i=1}^N f(x,\bsy^i) + \frac{\lambda_k}2  \frac1N\sum_{i=1}^N g(x,\bsy^i). 
\end{align*}

By the strong convexity of $\Psi_N^k$ it follows that
\begin{align*}
\|x_N^k - x_\infty^k\|_{{\mathcal X}}^2 &\le \frac{1}{m} \langle x_N^k-x_\infty^k, {\nabla_x\Psi_N^k(x_N^k )}- \nabla_x \Psi_N^k(x_\infty^k)\rangle_{{\mathcal X}}\\ &=  \frac{1}{m}\langle x_N^k-x_\infty^k, \nabla_x\Psi^k(x_\infty^k )- \nabla_x\Psi_N^k(x_\infty^k)\rangle_{{\mathcal X}}
\end{align*}
using the stationarity of the minimizers.
Applying the Cauchy--Schwarz inequality leads to
$$ \|x_N^k - x_\infty^k\|_{{\mathcal X}} \le \frac{1}{m}\|\nabla\Psi_N^k(x_\infty^k)-\nabla_x\Psi^k(x_\infty^k )\|_{{\mathcal X}}.$$
Next, we note that for $\psi^k(x):=f(x,\bsy)+\frac{\lambda_k}2g(x,\bsy)$ {there holds}
\begin{equation*}
\begin{split}
\|\nabla\Psi_N^k(x_\infty^k)-\nabla_x\Psi^k(x_\infty^k )\|_{{\mathcal X}}^2 = \Tr\bigg( &\bigg(\frac1N\sum_{i=1}^N {\nabla_x}\psi^k(x_\infty^k,\bsy^i) - \E_\bsy[{\nabla_x}\psi^k(x_\infty^k,\bsy)]\bigg)\\ \quad &{\otimes} \bigg(\frac1N\sum_{i=1}^N {\nabla_x}\psi^k(x_\infty^k,\bsy^i) - \E_\bsy[{\nabla_x}\psi^k(x_\infty^k,\bsy)]\bigg)\bigg)
\end{split}
\end{equation*}
and by taking the expectation
$$ \E[\|\nabla\Psi_N^k(x_\infty^k)-\nabla_x\Psi^k(x_\infty^k )\|_{{\mathcal X}}^2] = \frac1N\Tr(\rev{\Cov_\bsy}(\nabla_x \psi^k(x_\infty^k,\bsy))).$$
{There holds
\begin{align*} \Tr(\rev{\Cov_\bsy}(\nabla_x \psi^k(x,\bsy)))
&= \Tr(\rev{\Cov_\bsy}(\nabla_x f(x,\bsy)+\frac{\lambda_k}{2}\nabla_x g(x,\bsy) ))\\
&=\Tr(\rev{\Cov_\bsy}(\nabla_x f(x,\bsy))+\lambda_k \rev{\Cov_\bsy}(\nabla_x f(x,\bsy), \nabla_x g(x,\bsy))\\\
&\quad+ \frac{\lambda_k^2}{4}  \rev{\Cov_\bsy}(\nabla_x g(x,\bsy))\\
&\le \Tr(\max\{1,\frac{\lambda_k^2}{4}\} (\rev{\Cov_\bsy}(\nabla_x f(x,\bsy))+ 2\rev{\Cov_\bsy}(\nabla_x f(x,\bsy), \nabla_x g(x,\bsy))\\
&\quad+ \rev{\Cov_\bsy}(\nabla_x g(x,\bsy))\\
&=\max\{1,\frac{\lambda_k^2}{4}\}  \Tr(\rev{\Cov_\bsy}(\nabla_x \psi^{\rev{1}}(x,\bsy)))\, .
\end{align*}}
Finally, we obtain the bound
\begin{equation*}
\E[ \|x_N^k - x_\infty^k\|_{{\mathcal X}}^2]\le C_{\lambda_k}\frac1N\, , 
\end{equation*}
with $C_{\lambda_k} := \frac{\max\{1,\frac{\lambda_k^2}{4}\}\rev{\Tr(\Cov_\bsy(\nabla_x (f(x_\infty^k,\bsy)+\frac{\lambda_1}2 g(x_\infty^k,\bsy))))}}{m^2}$.
\end{proof}

The proof above reveals the dependence of the variance on $\lambda_k$. \rev{Thus, w}e expect that the number of data points needs to increase in order to counteract this effect, cp.~Theorem \ref{thm:consistency}.
\subsection{Convergence of pERM to cRM}
Finally, we are ready to prove consistency in the sense that solutions of the pERM converge to solutions of the original cRM. We can use Lemma~\ref{lem:pERM_pRM} and Lemma~\ref{lem:pRM_CRM} by applying
\begin{equation*}
\E[\| (z_N^k, \theta_N^k)-(z_\infty^\ast,\theta_\infty^\ast)\|_{{\mathcal X}}^2] \le 2\underbrace{\E[\| (z_N^k, \theta_N^k)-(z_\infty^k,\theta_\infty^k)\|_{{\mathcal X}}^2]}_{\text{pERM to pRM}} + 2 \underbrace{\|(z_\infty^k,\theta_\infty^k)-(z_\infty^\ast,\theta_\infty^\ast)\|_{{\mathcal X}}^2}_{\text{pRM to cRM}}.
\end{equation*}

\begin{theorem}\label{thm:consistency}
Suppose $f$ and $g$ satisfy the assumptions of Lemma~\ref{lem:pRM_CRM} and Lemma~\ref{lem:pERM_pRM}. Then the solution $(z_N^k,\theta_N^k)$ is consistent in the sense that there exists $C_1,C_2>0$ such that
\begin{equation*}
\E[\|(z_N^k,\theta_N^k) - (z_\infty^\ast,\theta_\infty^\ast)\|_{{\mathcal X}}^2]\le \frac{C_1}{\lambda_k^2}+\frac{C_2\cdot \lambda_k^2}{N}.
\end{equation*}
Moreover, for the choice $\lambda_k = \lambda(N) = N^{\frac14}$ the solution $(z_N^k,\theta_N^k)$ is consistent with 
\begin{equation*}
\E[\|(z_N^k,\theta_N^k) - (z_\infty^\ast,\theta_\infty^\ast)\|_{{\mathcal X}}^2]\le \frac{C}{\sqrt{N}},
\end{equation*}
for some constant $C>0$.
\end{theorem}

For a surrogate that is linear in its parameters, i.e., $u(\theta,\rev{\bsy}) = B\rev{(\bsy)} \theta$, the first assumption of Theorem \ref{thm:penalty_conv} (and thus Lemma \ref{lem:pRM_CRM}) follows from the strict convexity of $f$. The second assumption is clearly satisfied since for a linear surrogate,~\rev{as}~the constraint $h$ is linear and hence the objective $f$ is quadratic. The third assumption is true if we have for all $y\in L^2_\mu(U,W')$ that
\begin{align*}
    \mathbb{E}[\|(\nabla_{(\theta,z)} h(\theta^\ast,z^\ast))^\ast y\|^2_{\calX}]
    &\geq c \|y\|^2_{L^2_\mu(U,W')}\,.
\end{align*}
In the setting with linear surrogate the operator $(\nabla_{(\theta,z)} h(\theta^\ast,z^\ast))^\ast: L^2_\mu(U,W') \to \calX$ simplifies to $((\calA(\bsy) B\rev{(\bsy)})^\ast, -\calB^\ast)^\top$, such that we have $\mathbb{E}[\|((\calA\rev{(\bsy)} B\rev{(\bsy)})^\ast, -\calB^\ast)^\top y\|_\calX^2] = \mathbb{E}[\|(\calA\rev{(\bsy)} B\rev{(\bsy)})^\ast y\|_{\mathbb{R}^d}^2 + \|\calB^\ast y\|_{\calZ}^2] \geq \mathbb{E}[(a_{\min}^2 \sigma_{\min}(B\rev{(\bsy)} (B\rev{(\bsy)})^*) + \sigma_{\min}(\calB\calB^*)) \|y\|_{L_{\mu}^2(U,W')}^2]$. Here $\sigma_{\min}(\calB\calB^*)>0$ since $\calB$ has bounded inverse. 
Further,~\rev{due to} the linearity of the constraint, the Hessian of the Lagrangian simplifies to the Hessian of the objective function~$\nabla_{\rev{(\theta,z)}}^2f(\theta,z)= \diag{(\mathbb{E}[(B\rev{(\bsy)}){^*} B\rev{(\bsy)}], \alpha \cdot \calB^*\calB)}$. The fourth condition is thus satisfied if $\alpha >0$ and $\sigma_{\min}(\mathbb{E}[(B\rev{(\bsy)}){^*} B\rev{(\bsy)}]) \geq M$ for some $M>0$. If $\sigma_{\min}(\mathbb{E}[(B\rev{(\bsy)}){^*} B\rev{(\bsy)}]) = 0$ the fourth condition can still be satisfied by introducing a quadratic penalty on the surrogate parameters in the objective function, cp.~\eqref{eq:pERM}.

\section{Stochastic gradient descent for pRM problems}\label{sec:SGD}

In order to solve the pRM problem we propose to apply the stochastic gradient descent (SGD) method. This means, instead of solving the pERM problem offline for large but fixed number of data $N$, we solve the pRM online. \rev{W}e further propose to adaptively increase the penalty parameter $\lambda_k$ within the SGD. 

We firstly formulate a general convergence result for the penalized SGD method, which we then apply in order to verify the convergence in the setting of our PDE-constrained optimization problem given by the cRM problem.

\begin{algorithm}
\caption{Penalized stochastic gradient descent method with adaptive penalty parameter.}\label{alg:pSGD}
\begin{algorithmic}[1]
\Require $x_0$, $\beta = (\beta_k)_{k=1}^n$, $(\lambda_k)_{k=1}^n$, i.i.d.~sample $(\bsy^k)_{k=1}^n \sim\bsy.$ 
\For{$k=0,1,\dots,n-1$}
  \State $x_{k+1} = x_k - \beta_k \nabla_x [f(x_k,\bsy^k) + \lambda_k g(x_k,\bsy^k)]$
\EndFor
\end{algorithmic}
\end{algorithm}

The sequence of step sizes $\beta$ is assumed to satisfy the Robbins-Monro condition
\begin{equation}\label{eq:RM}
\sum_{\rev{j=0}}^\infty \beta_k = \infty,\quad \sum_{\rev{j=0}}^\infty \beta_k^2 <\infty,
\end{equation}
which means that $\beta_k$ converges to zero, but not too fast \cite{RM51}.  In the following theorem we present sufficient conditions under which the resulting estimate $x_n$ from Algorithm~\ref{alg:pSGD} converges to the solution of the pRM with penalty parameter choice $\bar \lambda\gg0$, i.e., to
\begin{equation*}
x^\ast \in \argmin_{x\in\R^n\times\R^d}\ \Psi_{\bar\lambda}(x),\quad \Psi_{\bar\lambda}(x):=\E_\bsy[f(x,\bsy) + \frac{\bar\lambda}2 g(x,\bsy)].
\end{equation*}
Due to Lemma~\ref{lem:pRM_CRM} the resulting limit point will be close \rev{to} the optimal solution of the original cRM.
\begin{theorem}\label{thm:SGDconvergence}
We assume that the objective function satisfies
\begin{equation}\label{eq:convexity_ass}
{\langle x-x^\ast, \nabla_x \Psi_{\bar\lambda}(x)\rangle_{\mathcal X}>c\|x-x^\ast\|_{\mathcal X}^2}
\end{equation}
for all {$x\in \mathcal X$} 
and some $c>0$ and that for each $\lambda_k$ we have
\begin{equation}\label{eq:variance_bound}
{\E_\bsy[\|\nabla_x [f(x,\bsy) + \lambda_k g(x,\bsy)\rev{]}\|_{\calX}^2] < {a_k + b_k \|x-x^\ast\|_{\mathcal X}^2}},
\end{equation}
where $(a_k)$ and $(b_k)$ are monotonically increasing with $a_0,b_0>0$ and $a_k\le \bar a,\ b_k\le \bar b$.
Furthermore, we assume that the discrepancy of the penalized stochastic gradients can be bounded {locally} by
\begin{equation}\label{eq:conv_grad}
{\sup_{x\in \mathcal X,\ \|x\|_{\mathcal X}\le R}}\ \|\E_\bsy[(\lambda_k-\bar\lambda) \nabla_x g(x,\bsy)]\|_{\calX}^2 \le \kappa_1(R) |\lambda_k - \bar\lambda|^2,
\end{equation}
for some $\kappa_1(R)>0$, $R>0$.  Suppose that $|\lambda_k-\bar\lambda|^2$ is monotonically decreasing and {$\beta_k\le c/{b_k}$}, 
then it holds true that
\begin{equation*}
\E[{{\mathds{1}_{\rev{\{\|x_j\|_{\mathcal X}\le R,\ j\le k\}}}}} \|x_k-x^\ast\|_{\mathcal X}^2] \le \left(\E[\|x_0-x^\ast\|_{\mathcal X}^2 + 2\bar a\sum_{j=1}^\infty \beta_j^2\right) C_n + {\frac{2\kappa_1(R)}{c^2}}|\lambda_0-\bar\lambda|^2,
\end{equation*}
with
\begin{equation*}
C_n := \min_{k\le n}\max \{ \prod_{j=k+1}^n (1-c\beta_j),\frac{\bar a}{c}\beta_k\}
\end{equation*}
converging to zero for $n\to\infty$. Further,  for an adaptive choice of the penalty parameter $\lambda_k$ such that ${\frac{2\kappa_1(R)}{c^2}} |\lambda_k-\bar\lambda|^2 \le D\beta_k$ we obtain
\begin{equation*}
\E[{{\mathds{1}_{\rev{\{\|x_j\|_{\mathcal X}\le R,\ j\le k\}}}}}\|x_k-x^\ast\|_{\mathcal X}^2] \le \left(\E[\|x_0-x^\ast\|_{\mathcal X}^2 + 2(\bar a+\frac{D}{c})\sum_{j=1}^\infty \beta_j^2\right) C_n.
\end{equation*}

\end{theorem}
\begin{proof} 
We denote $\Delta_{k} = x_k-x^\ast$ and write the SGD update step as
\begin{equation*}
\Delta_{k+1} = \Delta_k -\beta_k  \nabla_x \Psi_{\bar \lambda}(x_k) + \beta_k\delta_k +\beta_k \xi_k
\end{equation*}
with 
\[\delta_k = (\nabla_x \Psi_{\bar\lambda}(x_k) - \nabla_x \Psi_{\lambda_k}(x_k)),\quad \xi_k = (\nabla_x \Psi_{\lambda_k}(x_k) - \nabla_x (f(x_k,\bsy^k)+\lambda_k g(x_k,\bsy^k))).\]
We note that
\begin{align*}
\|-\nabla_x \Psi_{\bar\lambda}(x_k) + \delta_k\|_{\mathcal X}^2 + \E[\|\xi_k\|_{\mathcal X}^2\mid \mathcal F_{k}] &=  \E[\|\nabla_x (f(x_k,\bsy) +\lambda_k g(x_k,\bsy))\|_{\calX}^2].
\end{align*}
with $\mathcal F_k = \sigma(\bsy^m, \ m\le \rev{k-1})$.  
Denoting {$\mathds{1}_{R}^k := \mathds{1}_{\rev{\{\|x_j\|_{\mathcal X}\le R,\ j\le k\}}}$}, we can bound the increments by
\begin{align*}
\E[\mathds{1}_{R}^{k+1}\|\Delta_{k+1}\|_{\mathcal X}^2\mid \mathcal F_k] &\le \mathds{1}_R^k \Big(\|\Delta_k\|_{\mathcal X}^2 - 2\beta_k \langle \Delta_k, \nabla_x \Psi_{\bar\lambda}(x_k)-\delta_k\rangle_{\mathcal X}\\ &\quad +\beta_k^2 \|-\nabla \Psi_{\bar\lambda}(x_k) +\delta_k\|_{\mathcal X}^2 +\beta_k^2\E[\|\xi_k\|_{\mathcal X}^2\mid \mathcal F_{k}]\Big)\\
&\le \mathds{1}_R^k \Big(\|\Delta_k\|_{\mathcal X}^2 - 2\beta_k c\|\Delta_k\|_{\mathcal X}^2 + 2\beta_k\|\Delta_k\|_{\mathcal X} \|\delta_k\|_{\mathcal X} \\&\quad+\beta_k^2  \E[\|\nabla_x (f(x_k,\bsy^k) +\lambda_k g(x_k,\bsy^k))\|_{\mathcal X}^2]\Big),
\end{align*}
where we have used $\mathds{1}_{R}^{k+1}\le \mathds{1}_{R}^{k}$, the convexity assumption \eqref{eq:convexity_ass} and the Cauchy--Schwarz inequality. Using 
\eqref{eq:variance_bound} and applying Young's inequality of the form $a\cdot b \le \varepsilon/2 a^2+2/\varepsilon b^2$ with $\varepsilon >0$
leads to
\begin{align*}
\E[\mathds{1}_R^{k+1} \|\Delta_{k+1}\|_{\mathcal X}^2\mid \mathcal F_k] &\le \mathds{1}_R^k \Big(\|\Delta_k\|_{\mathcal X}^2 - 2\beta_k c\|\Delta_k\|_{\mathcal X}^2 + 2\beta_k\frac{\varepsilon}2\|\Delta_k\|_{\mathcal X}^2+\beta_k\frac1{\varepsilon} \|\delta_k\|_{\mathcal X}^2\\&\quad+\beta_k^2a_k+\beta_k^2b_k\|\Delta_k\|_{\mathcal X}^2\Big)\\
&{\le\mathds{1}_R^k \Big( \|\Delta_k\|_{\mathcal X}^2(1-2\beta_k c + 2\beta_k\frac{\varepsilon}2+\beta_k^2b_k) + \beta_k(\frac{{\kappa_1(R)}}{\varepsilon} |\lambda_k-\bar\lambda|^2+\beta_k a_k)}\big),
\end{align*}
where we have used \eqref{eq:conv_grad}.
{We set $\varepsilon = \frac{c}2$ such that we obtain with $\beta_k\le \frac{c}{2b_k}$}
{
\begin{align*}
\E[\mathds{1}_R^{k+1}\|\Delta_{k+1}\|_{\mathcal X}^2\mid \mathcal F_k] &\le (1-c\beta_k)\mathds{1}_R^k \|\Delta_k\|_{\mathcal X}^2+\beta_k(\frac{2{\kappa_1(R)}}{c} |\lambda_k-\bar\lambda|^2 + \beta_k a_k)
\end{align*}
}

We apply the discrete Gronwall's inequality and obtain
\begin{equation}\label{eq:Gronwall}
\begin{split}
\E[\mathds{1}_R^n\|\Delta_{n}\|_{\mathcal X}^2]&\le \sum_{k=1}^n\left(\prod_{j=k+1}^n (1-c\beta_j)a_k\beta_k^2\right)
+\frac{2\kappa_1}c\sum_{k=1}^n\left(\prod_{j=k+1}^n (1-c\beta_j)|\lambda_k-\bar\lambda|^2\beta_k \right)\\ &\quad+\exp\left(-c\sum_{k=1}^n\beta_k\right)\E[\|\Delta_0\|_{\mathcal X}^2].
\end{split}
\end{equation}

First note that $|\lambda_k-\bar\lambda|^2\le|\lambda_0-\bar\lambda|^2$, then it holds true that

\begin{align*}
\frac{2\kappa_1}c\sum_{k=1}^n\left(\prod_{j=k+1}^n (1-c\beta_j)|\lambda_k-\bar\lambda|^2\beta_k \right)&\le \frac{2\kappa_1}{c^2}|\lambda_0-\bar\lambda|^2 \sum_{k=1}^n\left(\prod_{j=k+1}^n (1-c\beta_j)c\beta_k \right)\\&\le \frac{2\kappa_1}{c^2}|\lambda_0-\bar\lambda|^2
\end{align*}
and with \eqref{eq:Gronwall} we obtain
\begin{equation*}
\begin{split}
\E[\mathds{1}_R^n\|\Delta_{n}\|_{\mathcal X}^2]&\le\bar{a} \sum_{k=1}^n\left(\prod_{j=k+1}^n (1-c\beta_j)\beta_k^2\right) + \frac{2\kappa_1}{c^2}\|\lambda_0-\bar\lambda\|^2+\exp\left(-c\sum_{k=1}^n\beta_k\right)\E[\|\Delta_0\|_{\mathcal X}^2].
\end{split}
\end{equation*}
Similarly, assuming that $\frac{2{\kappa_1(R)}}{c^2}|\lambda_k-\bar\lambda|^2\le D\beta_k$ gives
\begin{equation*}
\begin{split}
\E[\mathds{1}_R^n\|\Delta_{n}\|_{\mathcal X}^2]&\le(\bar{a}+D/c) \sum_{k=1}^n\left(\prod_{j=k+1}^n (1-c\beta_j)\beta_k^2\right) +\exp\left(-c\sum_{k=1}^n\beta_k\right)\E[\|\Delta_0\|_{\mathcal X}^2].
\end{split}
\end{equation*}

For the details of proving that
\[ \sum_{k=1}^n\left(\prod_{j=k+1}^n (1-c\beta_j)\beta_k^2\right)\le \left(\sum_{k=1}^\infty\beta_k^2\right) C_n,\quad \exp\left(-c\sum_{k=1}^n\beta_k\right)\le C_n \]
and the fact that $C_n$ converges to $0$ we refer the reader to the proof of Proposition~3.3 in \cite{CSTW21}.
\end{proof}

{
\begin{remark}
 {The concept of the proof of Theorem~\ref{thm:SGDconvergence} is similar to the proof of Proposition~3.3 in \cite{CSTW21}. While in \cite{CSTW21} one key challenge is to handle biased gradient approximations, the focus of the presented convergence result is to treat the incorporation of the controlled penalization through $\lambda_k$. \rev{Thus}, we need to redefine the (controlled) bias $\delta_k$ and noise $\xi_k$ of the stochastic gradient approximation in order to formulate a new discrete Gronwall-type argument.}
\end{remark}
\begin{remark}
We note that the restriction boundedness of $\|x_k\|_{\mathcal X}\le R$ is a techniqual reason for the proof and can be forced through a projection onto the ball {$B_R = \{x\in\mathcal X\mid \|x\|_{\calX}\le R\}$} by
{$$\mathcal P_R: \mathcal X\to B_R, \quad \text{with}\quad \mathcal P_R(x) = \arg\min_{x'\in B_R} \|x-x'\|_{\mathcal X}.$$
The projected stochastic gradient descent method then evolves through the update 
 $$x_{k+1} = \mathcal P_R\left(x_k - \beta_k \nabla_x [f(x_k,\bsy^k) + \lambda_k g(x_k,\bsy^k)]\right).$$
 The above proof remains the same since the projection operator is nonexpansive in the sense that
 \begin{align*}
 \|x_{k+1}-x^\ast\|_{\mathcal X} &= \|\mathcal P_R\left(x_k - \beta_k  \nabla_x [f(x_k,\bsy^k) + \lambda_k g(x_k,\bsy^k)]\right)-\mathcal P_R(x^\ast)\|_{\mathcal X}^2\\
 & \le \|x_k - \beta_k \nabla_x [f(x_k,\bsy^k) + \lambda_k g(x_k,\bsy^k)]-x^\ast\|_{\mathcal X}^2.
 \end{align*}}
 Moreoever, the presented convergence result in Theorem~\ref{thm:SGDconvergence} indicates how to control the ratio between the step sizes $(\beta_k)$ and the penalty parameters $(\lambda_k)$ based on the dependence of $\kappa(R)$ on $R>0$. 
\end{remark}
}

\subsection{Application to linear surrogate models}
In this section we verify that a surrogate, that is linear in its parameters, satisfies the assumptions of Theorem \ref{thm:SGDconvergence}. \rev{For this purpose}, we assume in this section that the surrogate is of the following form:
\begin{align}\label{eq:linearsurrogate}
u(\theta,\bsy) := B\rev{(\bsy)} \theta
\end{align}
for surrogate parameters $\theta \in \mathbb{R}^d$ and $\bsy$-dependent operator $B\rev{(\bsy)}:\mathbb{R}^d \to V$. The following two lemmas will help to prove this result:
\begin{lemma}\label{lemma:lemma01}
Let $g(x,\bsy) := \|e(u(\theta,\bsy),z)\|_{W'}^2$, with $e(u(\theta,\bsy),z) = A\rev{(\bsy)} B\rev{(\bsy)} \theta - \calB z$, where $x = (\theta,z)$, and with bounded largest eigenvalue $\sigma_{\max}(A\rev{(\bsy)}) \leq a_{\max}$ for all $\bsy \in U$. Then, there holds
\begin{align*}
\|\nabla_x g(x,\bsy)\|^2_{\calX} \leq 4(a_{\max}^2 \sigma_{\max}(B\rev{(\bsy)}(B\rev{(\bsy)})^*) + \sigma_{\max}(\calB\calB^*))\, g(x,\bsy)\,.
\end{align*}
\end{lemma}
\begin{proof} 
For all $\bsy \in U$ we have
\begin{align*}
\|\nabla_x g(x,\bsy)\|_{\calX}^2 &= \|2 (A\rev{(\bsy)} B\rev{(\bsy)})^* (A\rev{(\bsy)} B\rev{(\bsy)} \theta - \calB z)\|^2_{\mathbb{R}^d} + \|2 \calB^*(\calB z-A\rev{(\bsy)} B\rev{(\bsy)} \theta\|_{\calZ}^2\\
&\leq 4 ( a_{\max}^2 \sigma_{\max}(B\rev{(\bsy)}(B\rev{(\bsy)})^*) + \sigma_{\max}(\calB\calB^*)) \|A\rev{(\bsy)} B\rev{(\bsy)} \theta - \calB z\|_{W'}^2\\
&= 4 (a_{\max}^2 \sigma_{\max}(B\rev{(\bsy)}(B\rev{(\bsy)})^*) + \sigma_{\max}(\calB\calB^*)) g(x,\bsy)\,.
\end{align*}
\end{proof}
\begin{lemma}\label{lemma:lemma02}
Let $g(x,\bsy) := \|e(u(\theta,\bsy),z)\|_{W'}^2$, with $e(u(\theta,\bsy),z) = A\rev{(\bsy)} B\rev{(\bsy)} \theta - \calB z$, where $x = (\theta,z)$, and with bounded largest eigenvalue $\sigma_{\max}(A\rev{(\bsy)}) \leq a_{\max}$ for all $\bsy \in U$. Then, there holds for all $\bsy \in U$ that
\begin{align*}
g(x,\bsy) \leq 2 (a_{\max}^2 \sigma_{\max}((B\rev{(\bsy)})^* B\rev{(\bsy)})+\sigma_{\max}(\calB^*\calB)) \|x\|_{\calX}^2\,.
\end{align*}
\end{lemma}
\begin{proof}\rev{There holds}
\begin{align*}
g(x,\bsy) = \|A\rev{(\bsy)} B\rev{(\bsy)} \theta - \calB z\|_{W'}^2 &\leq 2( \|A\rev{(\bsy)} B\rev{(\bsy)} \theta\|_{W'}^2 + \|\calB z\|_{W'}^2) \\
&\leq 2( a_{\max}^2 \sigma_{\max}((B\rev{(\bsy)})^* B\rev{(\bsy)}) \|\theta\|_{\mathbb{R}^d}^2 + \sigma_{\max}(\calB^*\calB)) \|z\|_{\calZ}^2)\\
&\leq 2( a_{\max}^2 \sigma_{\max}((B\rev{(\bsy)})^* B\rev{(\bsy)}) +\sigma_{\max}(\calB^*\calB)) \|x\|_{\calX}^2\,.
\end{align*} 
\end{proof}

\begin{theorem}\label{thm:linearSGDconvergence}
Let $\alpha > 0$ and $\sigma_{\min} (\mathbb{E}[(B\rev{(\bsy)})^* B\rev{(\bsy)}])>0$, then a surrogate of the form \eqref{eq:linearsurrogate} satisfies the assumptions of Theorem \ref{thm:SGDconvergence}.
\end{theorem}

\begin{proof}
Firstly, we show that
\begin{align*}
    \langle x-x^\ast, \nabla_x \Psi_{\bar{\lambda}}(x)\rangle_{\calX} > c \|x-x^\ast\|_{\calX}^2
\end{align*}
is true for a constant $c>0$ and all $x \in \calX$. To verify this assumption, we first show that $\Psi_{\bar{\lambda}}$ is $c$-strongly convex. The $c$-strong convexity is equivalent to
\begin{align*}
\langle x-x^\ast, \nabla_x \Psi_{\bar{\lambda}}(x) - \nabla_{x} \Psi_{\bar{\lambda}}(x^\ast)\rangle_{\calX} \geq c \|x-x^\ast\|_{\calX}^2\,.
\end{align*}
Note that
\begin{align*}
    \nabla_{\theta} f(x,\bsy) = (\mathcal{Q}B\rev{(\bsy)})^* (\mathcal{Q}B\rev{(\bsy)}\theta - u_0)\quad \quad
    \nabla_{\theta} \frac{\lambda}{2} g(x,\bsy) = \lambda (A\rev{(\bsy)}B\rev{(\bsy)})^*(A\rev{(\bsy)}B\rev{(\bsy)}\theta - \calB z)\\
    \nabla_{z} f(x,\bsy) = \alpha z\quad \quad
    \nabla_{z} \frac{\lambda}{2} g(x,\bsy) = -\lambda \calB^* (A\rev{(\bsy)}B\rev{(\bsy)}\theta - \calB z)
\end{align*}
Using the linearity of the surrogate in its parameters, i.e., $u(\theta,\bsy) := B\rev{(\bsy)} \theta$, we obtain
\begin{align*}
&\langle x-x^\ast,  \nabla_x \Psi_{\bar{\lambda}}(x) - \nabla_{x^\ast} \Psi_{\bar{\lambda}}(x^\ast)\rangle_{\calX} \\
&= \langle (\theta-\theta^\ast,z-z^\ast), \mathbb{E}\big[ ( (\mathcal{Q} B\rev{(\bsy)})^* (\mathcal{Q}B\rev{(\bsy)}) + \bar{\lambda} (A\rev{(\bsy)} B\rev{(\bsy)})^* (A\rev{(\bsy)} B\rev{(\bsy)}) )(\theta-\theta^\ast) \\
&\quad -  \bar{\lambda} (A\rev{(\bsy)} B\rev{(\bsy)})^* \calB (z-z^\ast), (\alpha + \bar{\lambda} \calB^*\calB) (z-z^\ast)- \bar{\lambda} \calB^* A\rev{(\bsy)} B\rev{(\bsy)} (\theta-\theta^\ast) \big] \rangle_{\calX}\\
&= \langle \theta-\theta^\ast), \mathbb{E}[(\mathcal{Q} B\rev{(\bsy)})^* (\mathcal{Q}B\rev{(\bsy)})] (\theta-\theta^\ast)\rangle_{\mathbb{R}^d} + \bar{\lambda} \|A\rev{(\bsy)} B\rev{(\bsy)} (\theta-\theta^\ast)\|_{L^2_{\mu}(U,W')}^2\\
&\quad - \langle \theta-\theta^\ast, \bar{\lambda} (A\rev{(\bsy)} B\rev{(\bsy)})^* \calB (z-z^\ast)\rangle_{\mathbb{R}^d}\\
&\quad+ \alpha\|z-z^\ast\|_{\calZ}^2 + \bar{\lambda}\|\calB(z-z^\ast)\|_{L^2_{\mu}(U,W')}^2- \langle z-z^\ast, \bar{\lambda} \calB^* A\rev{(\bsy)} B\rev{(\bsy)} (\theta-\theta^\ast)\rangle_{\calZ}\\
&= \langle \theta-\theta^\ast, \mathbb{E}[(\mathcal{Q} B\rev{(\bsy)})^* (\mathcal{Q}B\rev{(\bsy)})] (\theta-\theta^\ast)\rangle_{\mathbb{R}^d} + \alpha \|z-z^\ast\|_{\calZ}^2 \\
&\quad + \bar{\lambda} \|\calB(z-z^\ast)- A\rev{(\bsy)} B\rev{(\bsy)} (\theta-\theta^\ast)\|_{L^2_{\mu}(U,W')}^2\\
&\geq \sigma_{\min}(\mathcal{Q}^*\mathcal{Q})\langle \theta-\theta^\ast, \mathbb{E}[(B\rev{(\bsy)})^* B\rev{(\bsy)}] (\theta-\theta^\ast)\rangle_{\mathbb{R}^d} + \alpha \|z-z^\ast\|_{\calZ}^2 
\geq c \|x-x^\ast\|_{\calX}^2 \,,
\end{align*}
where $c = \min{(\sigma_{\min}(\mathcal{Q}^*\mathcal{Q}) \sigma_{\min} (\mathbb{E}[(B\rev{(\bsy)})^* B\rev{(\bsy)}]), \alpha)}$. Note that $\sigma_{\min}(\mathcal{Q}'\mathcal{Q})>0$ since $\calQ$ is assumed to have a bounded inverse. Since $\alpha>0$ and $\sigma_{\min} (\mathbb{E}[(B\rev{(\bsy)})^* B\rev{(\bsy)}])>0$) by assumption, the assertion is true since $x^\ast$ is a stationary point of $\Psi_{\bar{\lambda}}(x)$.

Secondly, we show that
\begin{align*}
        \mathbb{E}[\| \nabla_x \big(f(x,\bsy)+ \lambda_k g(x,\bsy)\big)\|_{\calX}^2] \leq a_k + b_k \|x-x^\ast\|_{\calX}^2\,.
\end{align*}
For a stationary point $x^\ast$ of $\Psi_{\bar{\lambda}}(x)$ have that
\begin{align*}
0 = \nabla_x\big( f(x^\ast,\bsy) +\bar{\lambda} g(x^\ast,\bsy) \big)
\end{align*}
and thus
\begin{align*}
\|\nabla_x \big(f(x,\bsy)+ \lambda_k g(x,\bsy) \big)\|_{\calX}^2 &= \|\nabla_x \big(f(x,\bsy)+ \lambda_k g(x,\bsy) \big) - \nabla_x\big( f(x^\ast,\bsy) + \bar{\lambda} g(x^\ast,\bsy) \big)\|_{\calX}^2\\
&=\|\nabla_x f(x,\bsy) - \nabla_x f(x^\ast,\bsy) + \lambda_k \nabla_x g(x,\bsy) - \bar{\lambda} \nabla_x g(x^\ast,\bsy) \|_{\calX}^2\\
&\leq2\|\nabla_x (f(x,\bsy) - f(x^\ast,\bsy))\|_{\calX}^2\\
&\quad + 2\|\lambda_k \nabla_x g(x,\bsy) - \bar{\lambda} \nabla_x g(x^\ast,\bsy) \|_{\calX}^2\,.
\end{align*}
We have $\nabla_x \big(f(x,\bsy)-f(x^\ast,\bsy)\big) = \big( (\mathcal{Q}B\rev{(\bsy)})^* (\mathcal{Q}B\rev{(\bsy)}) (\theta - \theta^\ast), \alpha (z-z^\ast) \big)$ and thus
\[\mathbb{E}[2\|\nabla_x \big(f(x,\bsy)-f(x^\ast,\bsy)\big)\|_{\calX}^2] \leq \tilde{c}^2 \|x-x^\ast\|_{\calX}^2\,,\]
with $\tilde{c} =2 \mathbb{E}[\sigma_{\max}(\mathcal{Q}^*\mathcal{Q})\sigma_{\max} ((B\rev{(\bsy)})^* B\rev{(\bsy)})] + \alpha$.
Moreover, for the second summand we have
\begin{align*}
\nabla_x \big(\lambda_k g(x,\bsy) - \bar{\lambda}g(x^\ast,\bsy)\big) &= \big( 2(A\rev{(\bsy)} B\rev{(\bsy)})^* ((A\rev{(\bsy)} B\rev{(\bsy)}) (\lambda_k\theta - \bar{\lambda}\theta^\ast) \\
&\quad - \calB(\lambda_k z-\bar{\lambda}z^\ast)), 2\calB^*(\lambda_k z-\bar{\lambda}z^\ast) - A\rev{(\bsy)} B\rev{(\bsy)} (\lambda_k \theta- \bar{\lambda} \theta^\ast) \big)\\
&= \nabla_x g(\lambda_k x - \bar{\lambda} x^\ast, \bsy)\,.
\end{align*}
We can use this together with Lemma~\ref{lemma:lemma01} and Lemma~\ref{lemma:lemma02} to obtain
\begin{align*}
&\mathbb{E}[\|\nabla_x g(\lambda_k x - \bar{\lambda} x^\ast,\bsy)\|_{\calX}^2]\\ &\leq \mathbb{E}[4(a_{\max}^2 \sigma_{\max}(B\rev{(\bsy)}(B\rev{(\bsy)})^*) + \sigma_{\max}(\calB\calB^*))\, g(\lambda_k x - \bar{\lambda} x^\ast,\bsy)]\\
&\leq \mathbb{E}[4(a_{\max}^2 \sigma_{\max}(B\rev{(\bsy)}(B\rev{(\bsy)})^*) + \sigma_{\max}(\calB\calB^*))\\
&\quad\times 2 (a_{\max}^2 \sigma_{\max}((B\rev{(\bsy)})^* B\rev{(\bsy)})+\sigma_{\max}(\calB^*\calB)) \|\lambda_k x- \bar{\lambda}x^\ast\|_{\calX}^2]\\
&=\big(8 a_{\max}^4 \mathbb{E}[\sigma_{\max}((B\rev{(\bsy)})^* B\rev{(\bsy)})\sigma_{\max}(B\rev{(\bsy)}(B\rev{(\bsy)})^*)] + 8 \sigma_{\max}(\calB\calB^*) \sigma_{\max}(\calB^*\calB) \\
&\quad+ \sigma_{\max}(\calB\calB^*)(2a_{\max}^2 \mathbb{E}[\sigma_{\max}((B\rev{(\bsy)})^* B\rev{(\bsy)})]) + 4a_{\max}^2 \mathbb{E}[\sigma_{\max}(B\rev{(\bsy)}(B\rev{(\bsy)})^*)] \sigma_{\max}(\calB^*\calB)\big)\\
&\quad\times \|\lambda_k x- \bar{\lambda}x^\ast\|_{\calX}^2\\
&=\big(8 a_{\max}^4 \mathbb{E}[\sigma_{\max}((B\rev{(\bsy)})^* B\rev{(\bsy)})\sigma_{\max}(B\rev{(\bsy)}(B\rev{(\bsy)})^*)] + 8 \sigma_{\max}(\calB\calB^*) \sigma_{\max}(\calB^*\calB) \\
&\quad+ \sigma_{\max}(\calB\calB^*)(2a_{\max}^2 \mathbb{E}[\sigma_{\max}((B\rev{(\bsy)})^* B\rev{(\bsy)})]) + 4a_{\max}^2 \mathbb{E}[\sigma_{\max}(B\rev{(\bsy)}(B\rev{(\bsy)})^*)] \sigma_{\max}(\calB^*\calB)\big)\\
&\quad \times 2\big( \lambda_k^2 \|x-x^\ast\|_{\calX}^2 + (\lambda_k - \bar{\lambda})^2 \|x^\ast\|_{\calX}^2\big)\,.
\end{align*}
We conclude that
\begin{align*}
        \mathbb{E}[\| \nabla_x \big(f(x,\bsy)+ \lambda_k g(x,\bsy)\big)\|_{\calX}^2] \leq a_k + b_k \|x-x^\ast\|_{\calX}^2
    \end{align*}
holds for $a_0 \geq a_k = 2 C_{ab} (\lambda_k - \bar{\lambda})^2 \|x^*\|^{2}_{\calX}$ and $b_k = 2\tilde{c}^2 + 2C_{ab} \lambda_k^2$, where
\begin{align*}
C_{ab} &= \big(8 a_{\max}^4 \mathbb{E}[\sigma_{\max}((B\rev{(\bsy)})^* B\rev{(\bsy)})\sigma_{\max}(B\rev{(\bsy)}(B\rev{(\bsy)})^*)] + 8 \sigma_{\max}(\calB\calB^*) \sigma_{\max}(\calB^*\calB) \\
&\quad+ \sigma_{\max}(\calB\calB^*)(2a_{\max}^2 \mathbb{E}[\sigma_{\max}((B\rev{(\bsy)})^* B\rev{(\bsy)})]) + 4a_{\max}^2 \mathbb{E}[\sigma_{\max}(B\rev{(\bsy)}(B\rev{(\bsy)})^*)] \sigma_{\max}(\calB^*\calB)\big).
\end{align*}
Thirdly, {given some $R>0$} we show that
\begin{align*}
\sup_{x \in \calX, \|x\|_{\calX}\leq R} \|\mathbb{E}[(\lambda_k - \bar{\lambda}) \nabla_x g(x,\bsy)]\|^2 \leq \kappa_1(R) |\lambda_k - \bar{\lambda}|^2\,,
\end{align*}
for some {$\kappa_1(R) > 0$}. We observe that 
\begin{align*}
\|\mathbb{E}[(\lambda_k - \bar{\lambda}) \nabla_x g(x,\bsy)]\|_{\calX}^2 &\leq |\lambda_k - \bar{\lambda}|^2 \|\mathbb{E}[\nabla_x g(x,\bsy)]\|_{\calX}^2\\
&\leq |\lambda_k - \bar{\lambda}|^2 \mathbb{E}[\|\nabla_x g(x,\bsy)\|_{\calX}^2]\\
&\leq |\lambda_k - \bar{\lambda}|^2 \mathbb{E}[4(a_{\max}^2 \sigma_{\max}(B\rev{(\bsy)}(B\rev{(\bsy)})^*) + \sigma_{\max}(\calB\calB^*))\, g(x,\bsy)]\\
&\leq |\lambda_k - \bar{\lambda}|^2 \mathbb{E}[4(a_{\max}^2 \sigma_{\max}(B\rev{(\bsy)}(B\rev{(\bsy)})^*) + \sigma_{\max}(\calB\calB^*))\, \\ 
&\quad 2 (a_{\max}^2 \sigma_{\max}((B\rev{(\bsy)})^* B\rev{(\bsy)})+\sigma_{\max}(\calB^*\calB)) \|x\|_{\calX}^2]\,.
\end{align*}
This proves the claim for $\kappa_1(R) = \mathbb{E}[4(a_{\max}^2 \sigma_{\max}(B\rev{(\bsy)}(B\rev{(\bsy)})^*) + \sigma_{\max}(\calB\calB^*)) \linebreak 2(a_{\max}^2 \sigma_{\max}((B\rev{(\bsy)})^* B\rev{(\bsy)})+\sigma_{\max}(\calB^*\calB)) R^2]$.
\end{proof}

\begin{remark}
 We note that the assumption $\sigma_{\min} (\mathbb{E}[(B\rev{(\bsy)})^* B\rev{(\bsy)}])>0$ in Theorem~\ref{thm:linearSGDconvergence} can be dropped if a quadratic regularization on the surrogate parameters $\theta$ is employed. 
\end{remark}

\rev{
\subsection{Extension to non-convex objective functions}
In the following, we present an extending result for the convergence of Algorithm~\ref{alg:pSGD} in a non-convex setting. The presented convergence result is based on the super-martingale convergence theorem of Robbins-Siegmund \cite{RS1971}. The only assumption which we make is that the expectation functions $x\mapsto F(x):=\mathbb E_{\bsy}[f(x,\bsy)]$ and $x\mapsto G(x):=\mathbb E_{\bsy}[g(x,\bsy)]$ are smooth, i.e., we assume that the gradient $\nabla_x F$ is $L_1$-Lipschitz continuous and the gradient $\nabla_x G$ is $L_2$-Lipschitz continuous for $L_1,L_2>0$. A simple computation then gives that the expectation function $\Psi_{\lambda_k}$ is $L_1+\lambda_k^2 L_2$-smooth. Our convergence result reads as follows.
\begin{theorem}\label{thm:SGD_nonconvex}
    We assume that $g$ satisfies
    $\mathbb E_{\bsy}[g(x,\bsy)] \le a$
    for some $a>0$ and for all $x\in \mathcal X$, and there exist $L_1,L_2>0$, such that $\mathbb E_{\bsy}[f(x,\bsy)]$ is $L_1$-smooth and $\mathbb E_{\bsy}[g(x,\bsy)]$ is $L_2$-smooth. Moreover, suppose that $|\lambda_k-\bar\lambda|^2$ is monotonically decreasing, $\beta_k<\frac{1}{L_1+\lambda_k^2 L_2}$ and $(\beta_k)$ satisfies the Robbins-Monro condition \eqref{eq:RM}. Finally, we assume that there exist $b_1,b_2>0$ such that
    \[\mathbb E_{\bsy}[\|\nabla_x[f(x,\bsy)+\lambda_k g(x_k,\bsy)]-\mathbb E_{\bsy}\left[\nabla_x[f(x,\bsy)+\lambda_k g(x_k,\bsy)]\right]\|^2] \le b_1 + b_2 \lambda_k^2 \]
    Then it holds true that
    \[\liminf_{k\to\infty} \|\nabla_x \Psi_{\lambda_k}(x_k)\|^2 = 0\]
    almost surely. 
\end{theorem}
\begin{proof}
    Using that $x\mapsto \Psi_{\lambda_k}:=\mathbb E_{\bsy}[f(x,\bsy)+\lambda_k g(x,\bsy) ]$ is $L_1+\lambda_k^2 L_2$-smooth, we have the following iterative descent property \cite[Theorem~2.1.5]{N2018},
    \begin{align*}
        \Psi_{\lambda_k}(x_{k+1}) &\le \Psi_{\lambda_k}(x_k) + \langle \nabla_x\Psi_{\lambda_k}(x_k), x_{k+1}-x_k\rangle_{\mathcal X} + \frac{L_1+\lambda_k^2 L_2}{2}\|x_{k+1}-x_k\|_{\mathcal X}^2\\
        &= \Psi_{\lambda_{k-1}}(x_k) + [\Psi_{\lambda_k}-\Psi_{\lambda_{k-1}}](x_k) - \beta_k\langle \nabla_x\Psi_{\lambda_k}(x_k), \nabla_x[f(x,\bsy)+\lambda_k g(x_k,\bsy)]\rangle_{\mathcal X}\\
        &\quad +\beta_k^2 \frac{L_1+L_2\lambda_k^2}{2} \|\nabla_x[f(x,\bsy)+\lambda_k g(x_k,\bsy)]\|_{\mathcal X}^2\,. 
    \end{align*}
    We observe that
    \[\Psi_{\lambda_k}(x_k) - \Psi_{\lambda_{k-1}}(x_{k}) = (\lambda_k-\lambda_{k-1}) \mathbb E_{\bsy}[g(x_k,\bsy)]\, , \]
    with the convention $\lambda_{-1} = \lambda_0$, and conditioned on $\mathcal F_{k}=\sigma(\bsy^m, m\le k-1)$ we obtain 
    \[\mathbb E_\bsy[\langle \nabla_x\Psi_{\lambda_k}(x_k), \nabla_x[f(x,\bsy)+\lambda_k g(x_k,\bsy)]\rangle_{\mathcal X}] = \|\nabla_x\Psi_{\lambda_k}(x_k)\|_{\mathcal X}^2 \]
    as well as
    \begin{align*} 
    &\mathbb E_\bsy[\|\nabla_x[f(x,\bsy)+\lambda_k g(x_k,\bsy)]\|_{\mathcal X}^2 \mid \mathcal F_k]\\ &= \mathbb E[\|\nabla_x[f(x,\bsy)+\lambda_k g(x_k,\bsy)]-\mathbb E_{\bsy}[\nabla_x[f(x,\bsy)+\lambda_k g(x_k,\bsy)]]\|^2] + \|\nabla_x\Psi_{\lambda_k}(x_k)\|_{\mathcal X}^2\\
    &\le b_1 + b_2 \lambda_k^2 + \|\nabla_x\Psi_{\lambda_k}(x_k)\|_{\mathcal X}^2 \,. 
    \end{align*}
    Now, we define $\Delta_k = \Psi_{\lambda_{k-1}}(x_k)$ such that 
    \begin{align*} 
    \mathbb E_{\bsy}[\Delta_{k+1}\mid \mathcal F_k] &\le \Delta_k + (\lambda_k-\lambda_{k-1}) \mathbb E_{\bsy}[g(x_k,\bsy)] - \beta_k\left(1-\frac{L_1+\lambda_k^2 L_2}{2}\beta_k\right) \|\nabla_x\Psi_{\lambda_k}(x_k)\|_{\mathcal X}^2\\
    &\quad + \beta_k^2 \frac{L_1+\lambda_k^2 L_2}{2} (b_1+b_2\lambda_k^2)\\
    &\le \Delta_k + (\lambda_k-\lambda_{k-1}) \mathbb E_{\bsy}[g(x_k,\bsy)] - \frac{\beta_k}2 \|\nabla_x\Psi_{\lambda_k}(x_k)\|_{\mathcal X}^2\\
    &\quad + \beta_k^2 \frac{L_1+\lambda_k^2 L_2}{2} (b_1+b_2\lambda_k^2)\,.
    \end{align*}
    Note that we have the following properties
    \begin{itemize}
        \item $1-\frac{L_1+\lambda_k^2 L_2}{2}\beta_k\ge\frac12>0$ since $\beta_k < \frac{1}{L_1+\lambda_k^2 L_2}$,
        \item $0\le \sum_{k=0}^\infty (\lambda_k-\lambda_{k-1}) \mathbb E_{\bsy}[g(x_k,\bsy)] <\infty$ almost surely, since $\lambda_k-\lambda_{k-1}\ge0$, $\lambda_k\le \bar \lambda$ and $0\le g(x_k) \le a$.
        \item $0\le \sum_{k=0}^\infty \beta_k^2\frac{L_1+\lambda_k^2 L_2}{2} (b_1+b_2\lambda_k^2)<\infty$ almost surely, since $\lambda_k\le \bar \lambda$ and $\sum_{k=0}^\infty \beta_k^2<\infty$.
    \end{itemize}
    Finally, we are ready to apply 
    \cite[Theorem 1]{RS1971}, which gives
    \[ \sum_{k=0}^\infty \frac{\beta_k}2 \|\nabla_x\Psi_{\lambda_k}(x_k)\|_{\mathcal X}^2 <\infty \quad \text{almost surely}\,,\]
    and with the condition $\sum_{k=0}^\infty\beta_k = \infty$ implies the assertion 
    \[\liminf_{k\to\infty}\|\nabla_x\Psi_{\lambda_k}(x_k)\|_{\mathcal X}^2 = 0\quad \text{almost surely}\,.\]
\end{proof}
\begin{remark}
    In addition, using the triangle inequality, we obtain with Theorem~\ref{thm:SGD_nonconvex} that 
    \[\liminf_{k\to\infty}\|\nabla_x\Psi_{\bar \lambda}(x_k)\|_{\mathcal X}^2\le 2\liminf_{k\to\infty} \|\nabla_x\Psi_{\lambda_k}(x_k)\|_{\mathcal X}^2 + \lim_{k\to\infty} 2(\bar\lambda-\lambda_k) \|\nabla_x \mathbb E_{\bsy}[g(x_k,\bsy)]\|_{\mathcal X}^2 = 0\]
    almost surely, provided that $\|\nabla_x \mathbb E_{\bsy}[g(x_k,\bsy)]\|_{\mathcal X}^2$ is bounded.
\end{remark}
\begin{remark}
    We note that the above convergence result is a direct consequence of the Theorem of Robbins-Siegmund \cite[Theorem~1]{RS1971} and does not describe a rate of convergence. Indeed, the convergence of stochastic gradient methods in general can be very slow and is highly dependent on tuning of the step sizes. To obtain a rate of convergence one typically needs to include further assumptions on the loss function. For example, one can verify that $C_n$ in Theorem~\ref{thm:SGDconvergence} is of order $\frac{1}{n}$ which is a standard convergence behavior under smoothness and strong convexity \cite{Bottou2018}. More recently, it has been observed that similar convergence results can be shown under gradient dominance properties, sometimes also called Polyak-\L ojasiewicz condition, which can specifically be verified for over-parametrized (wide) neural networks when minimizing the non-linear least squares misfit \cite{LIU202285}. We expect, that in this scenario one can similarly verify a rate of convergence. However, such an analysis is beyond the scope of this paper.
\end{remark}
}
\section{Numerical Experiments}\label{sec:numerics}
The model problem in our numerical experiments is the Poisson equation, \eqref{eq:1.1_old}, \eqref{eq:1.2}~--~\eqref{eq:1.3}, on the unit square $D = (0,1)^2$. We use piecewise-linear finite elements on a uniform triangular mesh with meshwidth $h = 1/8$. The random input field is modelled as
 \[   a(\bsy,x) = a_0(x) + \sum_{j=1}^s y_j \frac{1}{(\pi^2 (k_j^2+\ell_j^2) + \tau^2)^\vartheta} \sin(\pi x_1 k_j) \sin(\pi x_2 \ell_j)\,,\]
where $a_0(x) = 10^{-5} + \|\sum_{j = 1}^s \frac{1}{(\pi^2 (k_j^2+\ell_j^2) + \tau^2)^\vartheta} \sin(\pi x_1 k_j) \sin(\pi x_2 \ell_j)\|_{L^\infty(D)}$, $s=4$, $\vartheta = 0.25$, $\tau = 3$,  $(k_j,\ell_j)_{j\in \{1,\ldots,s\}} \in \{1,\ldots,s\}^2$ and $y_j$ are i.i.d.~uniformly in $[-1,1]$ for all $j = 1,\ldots,s$. The variance of the resulting PDE solution $u(\bsy)$, with right-hand side $z(x) = x_2^2 - x_1^2$, is illustrated in Figure \ref{fig:var1} and Figure \ref{fig:var2}. The mean and standard deviation is estimated using $10^5$ Monte Carlo samples. 
        
\begin{figure}[htb]
    \begin{minipage}[t]{0.47\textwidth}
        \centering
        \includegraphics[width=\textwidth]{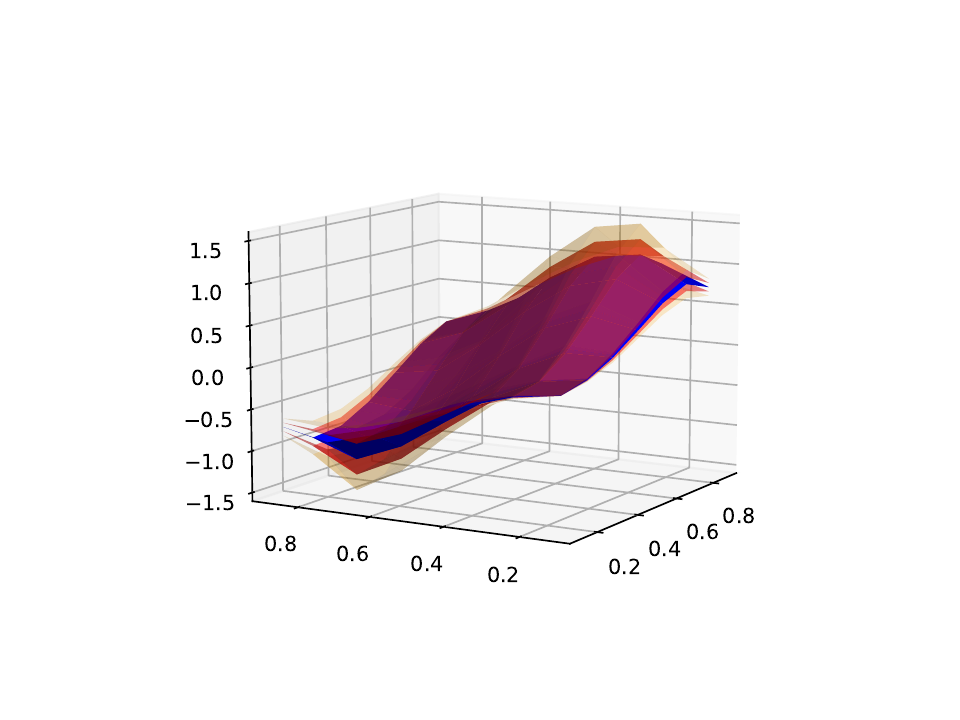}
        \caption{Mean of the states (blue) plus/minus 1 (red) and 2 (orange) standard deviations. Here only the values in the interior of the domain $D$ are plotted for better illustration.}
        \label{fig:var1}
    \end{minipage}
    \hfill
    \begin{minipage}[t]{0.47\textwidth} 
        \centering
        \includegraphics[width=\textwidth]{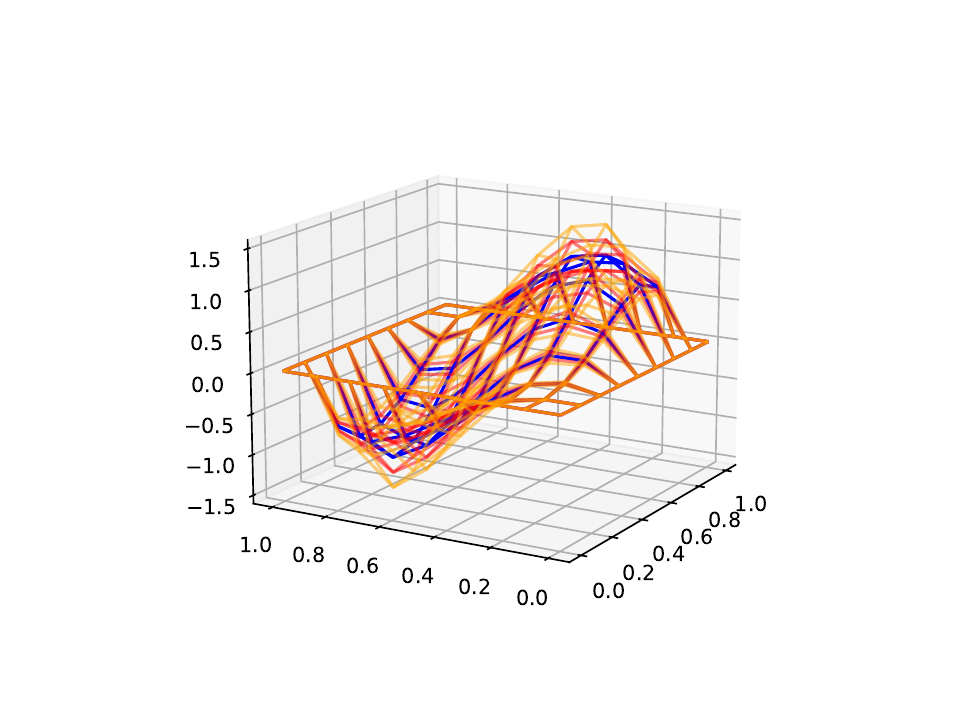}
        \caption{Mean of the states (blue) plus/minus 1 (red) and 2 (orange) standard deviations. The variance is zero on the boundary of the domain $\partial D$.}
        \label{fig:var2}
    \end{minipage}
\end{figure}
In the following numerical experiments we solve the pERM problem
\begin{align*}
    \min_{(z,\theta)}\, \frac{1}{N} \sum_{i=1}^N \|u(\theta,\bsy^{\rev{(i)}}) - u_0\|^2 + \frac{\alpha}{2} \|z\|^2 + \lambda_k \frac{1}{N} \sum_{i=1}^N \|A(\bsy^{\rev{(i)}}) u(\theta,\bsy^{\rev{(i)}}) - z\|^2\,.
\end{align*}
where $\alpha = 0.5$ and the target state $u_0$ is given as $u_0 = \Delta^{-1} 100(x_2^2-x_1^2)$. 

\subsection{\rev{Experiment 1}} 
\rev{In this section we shall assess the convergence results derived in Lemma~\ref{lem:pRM_CRM}, Lemma~\ref{lem:pERM_pRM}, and Theorem~\ref{thm:consistency} for a linear surrogate, namely a Legendre polynomial expansion of degree 2. We solve the optimization problems using the \texttt{scipy} implementation of the L-BFGS method with initial guess $z_0 = (0,\ldots,0) \in \mathbb{R}^n$ and $\theta_0 = (1,\ldots,1) \in \mathbb{R}^d$.}

First, we verify Lemma~\ref{lem:pERM_pRM}. To this end, we set $\lambda_k = 1$ for all $k$ and solve the pERM problem multiple times for increasing sample size $N = 2^k$, with $k = 1,\ldots,13$. The reference solution $(z_{\text{ref}}, \theta_{\text{ref}})$ is computed by using $N_{\text{ref}} = 2^{14}$ Monte Carlo samples. The observed rate in Figure \ref{fig:increasing_samplesize} aligns nicely with the predicted rate in Lemma \ref{lem:pERM_pRM}.
\begin{figure}[t]
\centering
    \begin{minipage}[t]{0.32\linewidth}
    \centering
        \includegraphics[width=\textwidth]{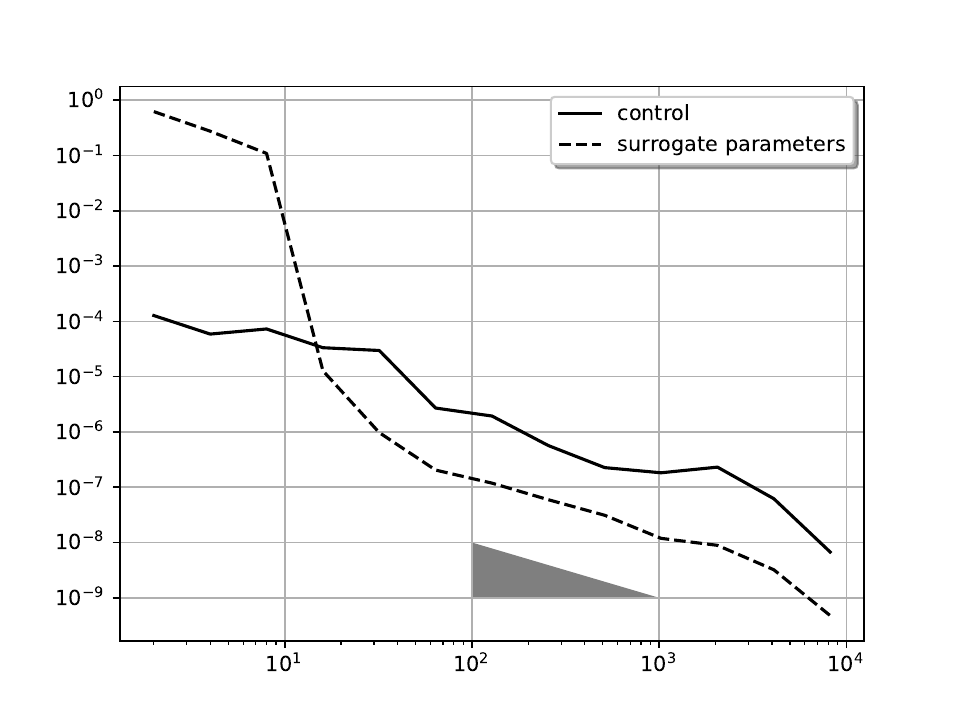}
        \caption{Convergence for increasing sample size. Squared error of the optimal controls $\|z - z_{\text{ref}}\|^2$ and squared error of the optimal surrogate parameters $\|\theta - \theta_{\text{ref}}\|^2$.}
        \label{fig:increasing_samplesize}
    \end{minipage}
    \hfill
    \begin{minipage}[t]{0.32\linewidth}
    \centering
    \includegraphics[width=\textwidth]{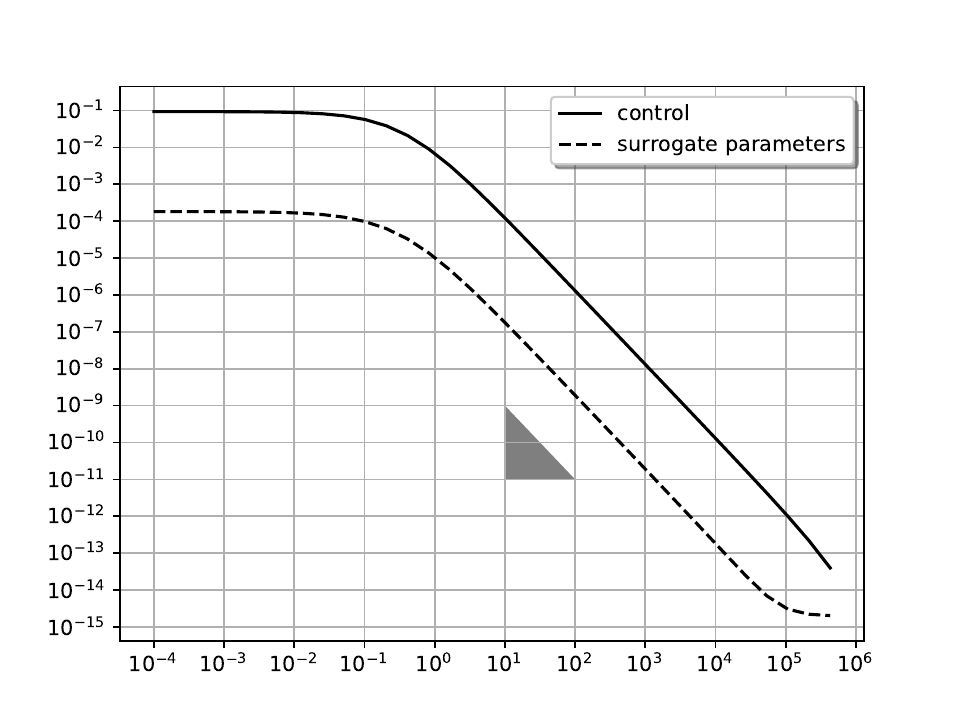}
        \caption{Convergence for increasing penalty parameter $\lambda_k$. Squared error of the optimal controls $\|z - z_{\text{ref}}\|^2$ and squared error of the optimal surrogate parameters $\|\theta - \theta_{\text{ref}}\|^2$.}
        \label{fig:increasing_penaltypara}
    \end{minipage}
    \hfill
    \begin{minipage}[t]{0.32\linewidth}
    \centering
        \includegraphics[width=\textwidth]{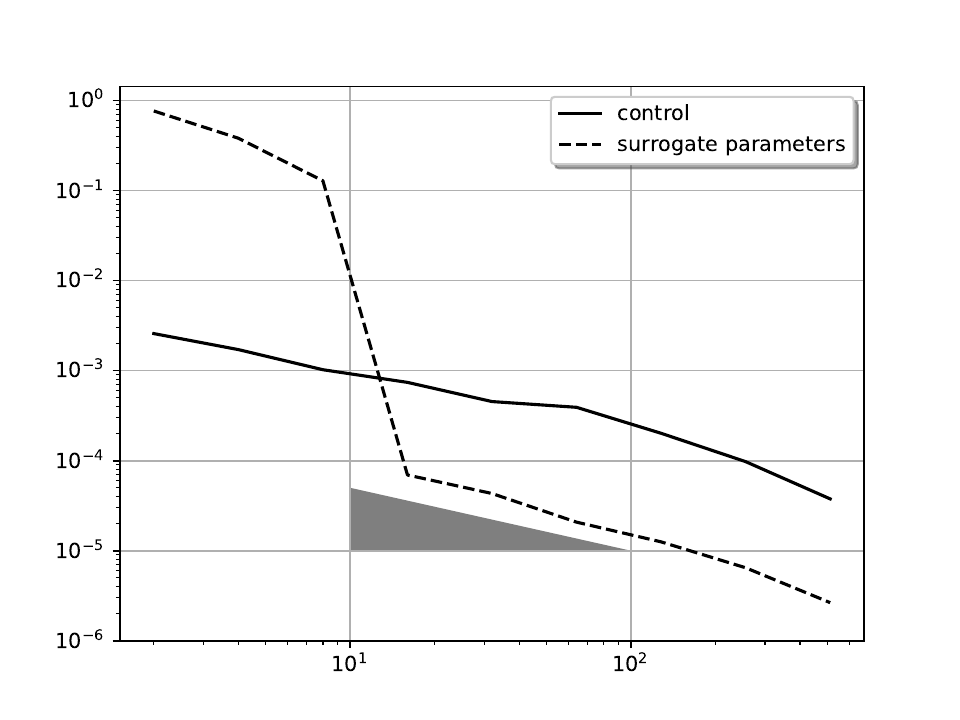}
        \caption{Convergence for simultaneous increasing of sample size and penalty parameter $\lambda_k = N^{\frac{1}{4}}$. Squared error of the optimal controls $\|z - z_{\text{ref}}\|^2$ and squared error of the optimal surrogate parameters $\|\theta - \theta_{\text{ref}}\|^2$.}
        \label{fig:increasing_simult}
    \end{minipage}
\end{figure}

Next, we verify Lemma~\ref{lem:pRM_CRM}. We fix the sample size $N = 100$ and solve the pERM problem for increasing penalty parameter $\lambda_k$. In Figure \ref{fig:increasing_penaltypara} we observe the rate predicted by Lemma \ref{lem:pRM_CRM}. Here the reference solution is computed for $\lambda_k \approx 1.7 \cdot 10^6$. For numerical stability we regularize the problem in this experiment by adding the term $10^{-5}\|\theta\|^2$ to the objective function of the pERM problem.

\rev{Finally, f}or the simultaneous increasing of the penalty parameter and the sample size, in Figure~\ref{fig:increasing_simult}, we solve the pERM for $N = 2^{k}$, with $k = 1,\ldots,9$ and balance $\lambda_k = N^{\frac{1}{4}}$. The reference solution is computed using $N=2^{11}$ points. The observed errors reflect the predicted rate in Theorem~\ref{thm:consistency}.

\subsection{\rev{Experiment 2}}

\rev{In this experiment we solve the pERM problem using Algorithm~\ref{alg:pSGD} with adaptively increasing penalty parameter $\lambda_k$ in accordance with Theorem~\ref{thm:SGDconvergence}.}

\begin{figure}[t]
\centering
    \begin{minipage}[t]{0.47\linewidth}
    \centering
        \includegraphics[width=\textwidth]{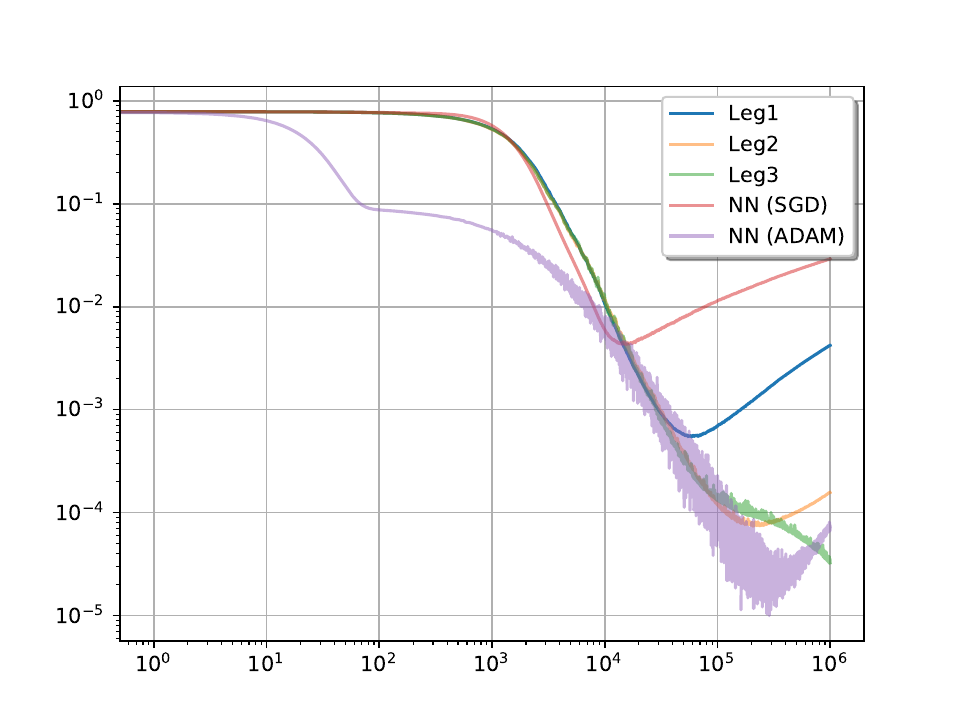}
        \caption{Mean squared error of control computed with surrogate and L-BFGS reference solution of the control $\|z-z_{\text{ref}}\|^2$}\label{fig:error_control}
    \end{minipage}
    \hfill
    \begin{minipage}[t]{0.47\linewidth}
    \centering
        \includegraphics[width=\textwidth]{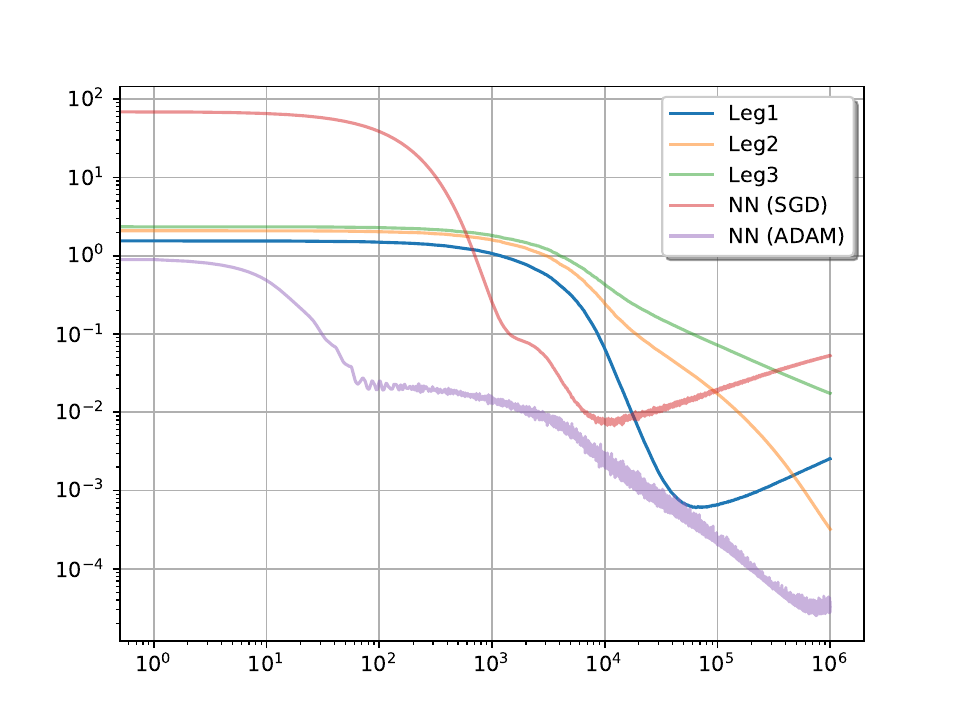}
        \caption{Mean squared error of surrogate and L-BFGS reference solution of the state $\mathbb{E}[\|u\rev{(\bsy)}_{\theta} - u_{\text{ref}}\rev{(\bsy)}\|^2]$}
        \label{fig:error_state}
    \end{minipage}
\end{figure}

\begin{figure}[t]
\centering
    \begin{minipage}[t]{0.47\linewidth}
    \centering
        \includegraphics[width=\textwidth]{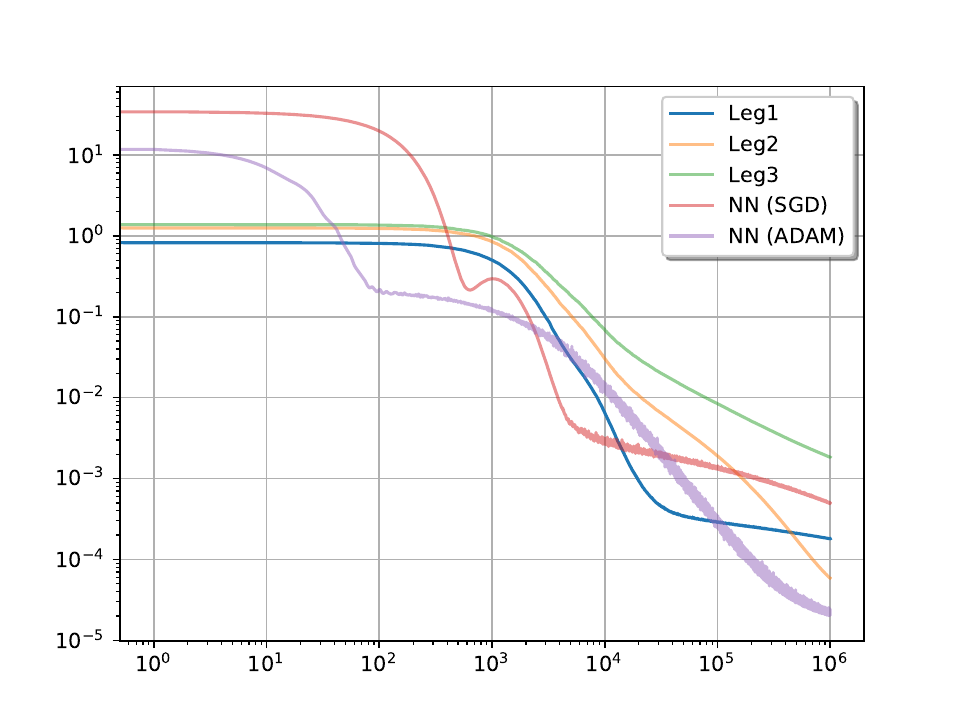}
        \caption{Mean squared residual $\mathbb{E}[\|A\rev{(\bsy)} u\rev{(\bsy)}_{\theta} - z\|^2]$}
        \label{fig:model_error}
    \end{minipage}
    \hfill
    \begin{minipage}[t]{0.47\linewidth}
    \centering
        \includegraphics[width=\textwidth]{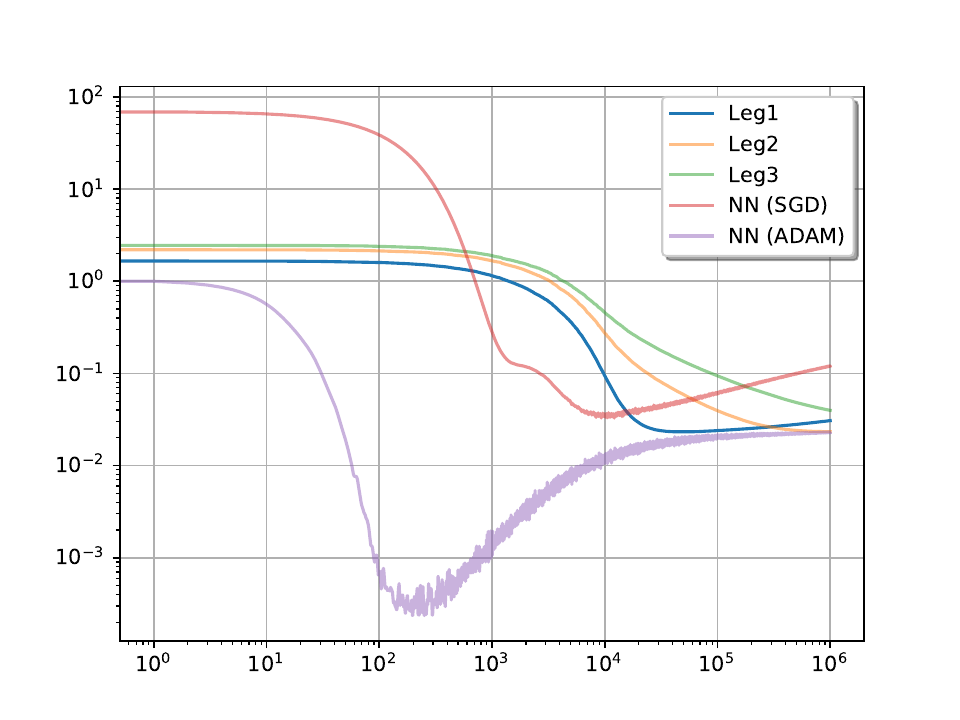}
        \caption{Mean squared error $\mathbb{E}[\|u\rev{(\bsy)}_\theta-u_0\|^2]$ of the surrogate $u_\theta\rev{(\bsy)}$ and the target state $u_0$}
        \label{fig:tracking_error}
    \end{minipage}
\end{figure}

\rev{
We compared two different surrogate models: the orthogonal Legendre polynomial expansion and a neural network. We emphasize that the convergence results are not verified for the neural network.}

The experiment is performed for Legendre polynomial expansions, which are linear in its parameters, of order $1,2$ and $3$. More precisely, recalling from \eqref{eq:Legendre}, we consider $u(\theta,\bsy) = \sum_{\bsnu \in \mathbb{N}_0^s}^{|\bsnu|=\ell} \theta_\bsnu P_\bsnu(\bsy)$ of degree $\ell = 1,2,3$, with $P_\bsnu = \prod_{k=1}^s P_{\nu_k}(y_k)$, where~$P_{\nu_k}$ is the k-th order Legendre polynomial. The number of parameters $\theta$ increases rapidly as the order of the polynomials increases. In fact, $\theta \in \mathbb{R}^{n_{\text{FEM}}\times n_{\text{Pol}}}$, where $n_{\text{FEM}}$ denotes the number of degree of freedoms of the finite element method and $n_{\text{Pol}}$ denotes the number of polynomials given by $n_{\text{Pol}} = \frac{(\ell + s)!}{\ell! s!}$, i.e., $n_{\text{Pol}} = 15$ if $\ell= 2$ and $n_{\text{Pol}} = 35$ if $\ell = 3$ for $s=4$. Consequently, the Legendre polynomial expansions have $245$, $735$, and $1715$ parameters to be determined during the optimization. 

\rev{As a nonlinear surrogate we have chosen to test a neural network, as defined in \eqref{eq:neuralnetwork} with $L = 5$ layers and of size $N_1 = 4$, $N_2,\ldots,N_4 = 9$ and $N_5 = 49$}
, i.e., with a total number of $715$ parameters. The activation function we are using is the sigmoid function $\sigma(x) := 1/(1+\exp{(-x)})$.
\rev{To solve the pERM problem with the neural network surrogate we use both, the penalized stochastic gradient descent with adaptive penalty parameter (see, Algorithm~\ref{alg:pSGD}), as well as the ADAM algorithm with increasing penalty parameter. We have chosen a smooth activation function. For results on optimization in the nonsmooth setting, we refer to \cite{christof2023identification,Constantin,doi:10.1137/18M1231559,dong2022descent,dong2022firstorder}}

For each of the surrogates considered, the error in the control decreases, see Figure \ref{fig:error_control}. Clearly, this error is bounded from below by the approximation properties of the surrogates. In Figure \ref{fig:error_state} we observe that the errors of the states decrease as well. The reference solution $(z_{\text{ref}}, u\rev{(\bsy)}_{\text{ref}})$ of the cRM problem is computed using the L-BFGS as implemented in \texttt{scipy}.
%
\rev{Furthermore, we observe in Figure~\ref{fig:model_error}} that the model error \rev{decreases}.

\section{Conclusions}
\label{sec:conclusions}
We proposed a flexible framework for the incorporation of machine learning approximations for optimization under uncertainty. The surrogate is trained only for the optimal control \rev{instead of for all possible control values}, i.e., no expensive offline training \rev{for the entire control space} is needed. \rev{This can significantly reduce the computational complexity of the underlying problem and gives hope that smaller surrogates can be used in the presented approach}. The numerical experiments show promising results and application to more complex optimization problems under uncertainty will be subject to future work.

We analyzed the stochastic gradient method for the optimization of the one-shot system. In more complex situations, gradients might not be available due to the use of black-box solvers or computational limits. We will explore the generalization of our work to derivative-free optimization techniques, in particular to Kalman based methods, in order to ensure applicability also in this setting.

\rev{Moreover, we have observed in our numerical experiments that the implementation of ADAM can lead to significant improvement of the performance compared to standard SGD (with increasing penalty parameter). It would be very interesting to analyse this behavior from a theoretical perspective. In particular, a natural starting point is to analyse the penalized SGD algorithm with additional incorporation of momentum such as Polyak's heavy ball method~\cite{POLYAK19641} or Nesterov's acceleration method~\cite{N2018}.}

The penalty approach, which can be motivated by the Bayesian ansatz \cite{GSW2020}, provides an algorithmic framework ensuring the feasibility of the control \rev{with respect} to the forward model. The Bayesian viewpoint will guide further work to adaptively learn the penalty function and incorporate model error.

\section*{Acknowledgments}
PG is grateful to the DFG RTG1953 ``Statistical Modeling of Complex Systems and Processes'' for partial funding of this research. CS acknowledges support from MATH+ project EF1-19: Machine Learning Enhanced Filtering Methods for Inverse Problems and EF1-20: Uncertainty Quantification and Design of Experiment for Data-Driven Control, funded by the Deutsche Forschungsgemeinschaft (DFG, German Research Foundation) under Germany's Excellence Strategy - The Berlin Mathematics Research Center MATH+ (EXC-2046/1, project ID: 390685689). The authors acknowledge support by the state of Baden-W\"urttemberg through bwHPC.


\bibliographystyle{siam}
\bibliography{references.bib}

\begin{thebibliography}{10}

\bibitem{doi:10.1137/16M106306X}
{\sc A.~Alexanderian, N.~Petra, G.~Stadler, and O.~Ghattas}, {\em Mean-variance
  risk-averse optimal control of systems governed by {PDE}s with random
  parameter fields using quadratic approximations}, SIAM/ASA J. Uncertain.
  Quantif., 5 (2017), pp.~1166--1192.

\bibitem{AUH}
{\sc A.~A. Ali, E.~Ullmann, and M.~Hinze}, {\em {M}ultilevel {M}onte {C}arlo
  analysis for optimal control of elliptic {PDE}s with random coefficients},
  SIAM/ASA J. Uncertain. Quantif., 5 (2017), pp.~466--492.

\bibitem{cite-key}
{\sc A.~Alla, M.~Hinze, P.~Kolvenbach, O.~Lass, and S.~Ulbrich}, {\em A
  certified model reduction approach for robust parameter optimization with
  {PDE} constraints}, Adv. Comput. Math., 45 (2019), pp.~1221--1250.

\bibitem{10.1093/imanum/drx052}
{\sc M.~Bachmayr, A.~Cohen, and W.~Dahmen}, {\em {Parametric PDEs: sparse or
  low-rank approximations?}}, IMA J. Numer. Anal., 38 (2017), pp.~1661--1708.

\bibitem{BHKS2021}
{\sc K.~Bhattacharya, B.~Hosseini, N.~B. Kovachki, and A.~M. Stuart}, {\em
  Model reduction and neural networks for parametric {PDEs}}, J. Comput. Math.,
  7 (2021), pp.~121--157.

\bibitem{volker}
{\sc A.~Borzi and V.~Schulz}, {\em Computational optimization of systems
  governed by partial differential equations}, Society for Industrial and
  Applied Mathematics, USA, 2012.

\bibitem{BSSW}
{\sc A.~Borzi, V.~Schulz, C.~Schillings, and G.~{von Winckel}}, {\em {O}n the
  treatment of distributed uncertainties in {PDE}-constrained optimization},
  GAMM-Mit., 33 (2010), pp.~230--246.

\bibitem{Bottou2018}
{\sc L.~Bottou, F.~E. Curtis, and J.~Nocedal}, {\em Optimization methods for
  large-scale machine learning}, SIAM Review, 60 (2018), pp.~223--311.

\bibitem{CSTW21}
{\sc N.~K. Chada, C.~Schillings, X.~T. Tong, and S.~Weissmann}, {\em
  Consistency analysis of bilevel data-driven learning in inverse problems},
  Commun. Math. Sci., 20 (2021), pp.~123--164.

\bibitem{CVG}
{\sc P.~Chen, U.~Villa, and O.~Ghattas}, {\em {T}aylor approximation and
  variance reduction for {PDE}-constrained optimal control under uncertainty},
  J. Comput. Phys., 385 (2019), pp.~163--186.

\bibitem{chen2021solving}
{\sc Y.~Chen, B.~Hosseini, H.~Owhadi, and A.~M. Stuart}, {\em Solving and
  learning nonlinear {PDE}s with {G}aussian processes}, J. Comput. Phys., 447
  (2021), p.~110668.

\bibitem{christof2023identification}
{\sc C.~Christof and J.~Kowalczyk}, {\em On the identification and optimization
  of nonsmooth superposition operators in semilinear elliptic {PDE}s}, 2023.
\newblock arXiv preprint.

\bibitem{Constantin}
{\sc C.~Christof, C.~Meyer, S.~Walther, and C.~Clason}, {\em Optimal control of
  a non-smooth semilinear elliptic equation}, Math. Control Relat. Fields, 8
  (2018), pp.~247--276.

\bibitem{CohenDeVore15}
{\sc A.~Cohen and R.~A. DeVore}, {\em Approximation of high-dimensional
  parametric {PDEs}}, Acta Numer., 24 (2015), pp.~1--159.

\bibitem{doi:10.1137/18M1231559}
{\sc Y.~Cui, Z.~He, and J.-S. Pang}, {\em Multicomposite nonconvex optimization
  for training deep neural networks}, SIAM J. Optim., 30 (2020),
  pp.~1693--1723.

\bibitem{dong2020optimization}
{\sc G.~{Dong}, M.~{Hinterm\"uller}, and K.~{Papafitsoros}}, {\em Optimization
  with learning-informed differential equation constraints and its
  applications}, ESAIM: COCV, 28 (2022), p.~3.

\bibitem{dong2022descent}
{\sc G.~Dong, M.~Hintermüller, and K.~Papafitsoros}, {\em A descent algorithm
  for the optimal control of {R}e{LU} neural network informed {PDE}s based on
  approximate directional derivatives}, 2022.
\newblock arXiv preprint.

\bibitem{dong2022firstorder}
{\sc G.~Dong, M.~Hintermüller, K.~Papafitsoros, and K.~Völkner}, {\em
  First-order conditions for the optimal control of learning-informed nonsmooth
  {PDE}s}, 2022.
\newblock arXiv preprint.

\bibitem{Gitta1}
{\sc M.~Geist, P.~Petersen, M.~Raslan, R.~Schneider, and G.~Kutyniok}, {\em
  Numerical solution of the parametric diffusion equation by deep neural
  networks}, J. Sci. Comput., 88 (2021), p.~22.

\bibitem{Gittelson2011}
{\sc C.~J. Gittelson}, {\em Adaptive Galerkin methods for parametric and
  stochastic operator equations}, dissertation, ETH Zurich, 2011.

\bibitem{DBLP:journals/simods/GuntherRSCG20}
{\sc S.~G\"unther, L.~Ruthotto, J.~B. Schroder, E.~C. Cyr, and N.~R. Gauger},
  {\em Layer-parallel training of deep residual neural networks}, SIAM J. Math.
  Data Sci., 2 (2020), pp.~1--23.

\bibitem{GK22}
{\sc P.~A. Guth and V.~Kaarnioja}, {\em Generalized dimension truncation error
  analysis for high-dimensional numerical integration: lognormal setting and
  beyond}, arXiv preprint,  (2022).

\bibitem{GKKSS21}
{\sc P.~A. Guth, V.~Kaarnioja, F.~Y. Kuo, C.~Schillings, and I.~H. Sloan}, {\em
  A quasi-{M}onte {C}arlo method for optimal control under uncertainty},
  SIAM/ASA J. Uncertain. Quantif., 9 (2021), pp.~354--383.

\bibitem{GKKSS22}
\leavevmode\vrule height 2pt depth -1.6pt width 23pt, {\em Parabolic
  {PDE}-constrained optimal control under uncertainty with entropic risk
  measure using quasi-{M}onte {C}arlo integration}, arXiv preprint,  (2022).

\bibitem{GSW2020}
{\sc P.~A. Guth, C.~Schillings, and S.~Weissmann}, {\em 14 {E}nsemble {K}alman
  filter for neural network-based one-shot inversion}, In: Optimization and
  Control for Partial Differential Equations: Uncertainty quantification, open
  and closed-loop control, and shape optimization, edited by R. Herzog and M.
  Heinkenschloss and D. Kalise and G. Stadler and E. Trélat, Berlin, Boston:
  De Gruyter (2022), pp.~393--418.

\bibitem{Han8505}
{\sc J.~Han, A.~Jentzen, and W.~E}, {\em Solving high-dimensional partial
  differential equations using deep learning}, Proceedings of the National
  Academy of Sciences, 115 (2018), pp.~8505--8510.

\bibitem{hinze2008optimization}
{\sc M.~Hinze, R.~Pinnau, M.~Ulbrich, and S.~Ulbrich}, {\em Optimization with
  PDE constraints}, vol.~23, Springer Science \& Business Media, 2008.

\bibitem{K16}
{\sc B.~Kaltenbacher}, {\em Regularization based on all-at-once formulations
  for inverse problems}, SIAM J. Numer. Anal., 54 (2016), pp.~2594--2618.

\bibitem{K17}
\leavevmode\vrule height 2pt depth -1.6pt width 23pt, {\em All-at-once versus
  reduced iterative methods for time dependent inverse problems}, Inverse
  Probl., 33 (2017), p.~064002.

\bibitem{KHRB}
{\sc D.~P. Kouri, M.~Heinkenschloss, D.~Ridzal, and B.~G. {van Bloemen
  Waanders}}, {\em {A} {t}rust-{r}egion {a}lgorithm with {a}daptive
  {s}tochastic {c}ollocation for {PDE} {o}ptimization under {u}ncertainty},
  SIAM J. Sci. Comput., 35 (2013), pp.~A1847--A1879.

\bibitem{Kouri2018}
{\sc D.~P. Kouri and A.~Shapiro}, {\em Optimization of PDEs with Uncertain
  Inputs}, Springer New York, New York, NY, 2018, pp.~41--81.

\bibitem{doi:10.1137/140954556}
{\sc D.~P. Kouri and T.~M. Surowiec}, {\em Risk-averse {PDE}-constrained
  optimization using the conditional value-at-risk}, SIAM J. Optim., 26 (2016),
  pp.~365--396.

\bibitem{Drew2018}
\leavevmode\vrule height 2pt depth -1.6pt width 23pt, {\em Existence and
  optimality conditions for risk-averse {PDE}-constrained optimization},
  SIAM/ASA J. Uncertain. Quantif., 6 (2018), pp.~787--815.

\bibitem{Kunoth}
{\sc A.~Kunoth and C.~Schwab}, {\em Analytic regularity and {GPC} approximation
  for control problems constrained by linear parametric elliptic and parabolic
  {PDE}s}, SIAM J. Control Optim., 51 (2013), pp.~2442--2471.

\bibitem{kuo2017multilevel}
{\sc F.~Y. Kuo, R.~Scheichl, C.~Schwab, I.~H. Sloan, and E.~Ullmann}, {\em
  Multilevel quasi-{M}onte {C}arlo methods for lognormal diffusion problems},
  Math. Comput., 86 (2017), pp.~2827--2860.

\bibitem{Gitta2}
{\sc G.~Kutyniok, P.~Petersen, M.~Raslan, and R.~Schneider}, {\em A theoretical
  analysis of deep neural networks and parametric {PDEs}}, Constr. Approx.,
  (2021).

\bibitem{LIU202285}
{\sc C.~Liu, L.~Zhu, and M.~Belkin}, {\em Loss landscapes and optimization in
  over-parameterized non-linear systems and neural networks}, Applied and
  Computational Harmonic Analysis, 59 (2022), pp.~85--116.
\newblock Special Issue on Harmonic Analysis and Machine Learning.

\bibitem{LJK2019}
{\sc L.~Lu, P.~Jin, and G.~E. Karniadakis}, {\em Deeponet: Learning nonlinear
  operators for identifying differential equations based on the universal
  approximation theorem of operators}, arXiv preprint,  (2019).

\bibitem{doi:10.1137/19M1263121}
{\sc J.~Milz and M.~Ulbrich}, {\em An approximation scheme for distributionally
  robust nonlinear optimization}, SIAM J. Optim., 30 (2020), pp.~1996--2025.

\bibitem{N2018}
{\sc Y.~Nesterov}, {\em {Lectures on Convex Optimization}}, Springer Cham,
  2nd~ed., 2018.

\bibitem{OPS19}
{\sc J.~A.~A. Opschoor, P.~C. Petersen, and C.~Schwab}, {\em Deep {R}e{LU}
  networks and high-order finite element methods}, Anal. Appl.,  (2019).

\bibitem{POLYAK19641}
{\sc B.~Polyak}, {\em Some methods of speeding up the convergence of iteration
  methods}, USSR Computational Mathematics and Mathematical Physics, 4 (1964),
  pp.~1--17.

\bibitem{POLYAK19711}
{\sc B.~T. Polyak}, {\em The convergence rate of the penalty function method},
  USSR Comput. Math. \& Math. Phys., 11 (1971), pp.~1--12.

\bibitem{RPK17}
{\sc M.~Raissi, P.~Perdikaris, and G.~E. Karniadakis}, {\em Physics informed
  deep learning ({P}art {I}): {D}ata-driven solutions of nonlinear partial
  differential equations}, arXiv preprint,  (2017).

\bibitem{RPK17_2}
\leavevmode\vrule height 2pt depth -1.6pt width 23pt, {\em Physics informed
  deep learning ({P}art {II}): {D}ata-driven discovery of nonlinear partial
  differential equations}, arXiv preprint,  (2017).

\bibitem{RPK19}
\leavevmode\vrule height 2pt depth -1.6pt width 23pt, {\em Physics-informed
  neural networks: A deep learning framework for solving forward and inverse
  problems involving nonlinear partial differential equations}, J. Comput.
  Phys., 378 (2019), pp.~686--707.

\bibitem{RM51}
{\sc H.~Robbins and S.~Monro}, {\em A stochastic approximation method}, Ann.
  Math. Stat., 22 (1951), pp.~400 -- 407.

\bibitem{RS1971}
{\sc H.~Robbins and D.~Siegmund}, {\em A convergence theorem for non negative
  almost supermartingales and some applications}, in Optimizing Methods in
  Statistics, J.~S. Rustagi, ed., Academic Press, 1971, pp.~233--257.

\bibitem{rozza:hal-01722593}
{\sc G.~Rozza, D.~B.~P. Huynh, and A.~T. Patera}, {\em Reduced basis
  approximation and a posteriori error estimation for affinely parametrized
  elliptic coercive partial differential equations}, {Arch. Comput. Methods
  Eng.}, 15 (2008), pp.~229 -- 275.

\bibitem{SCHILLINGS201178}
{\sc C.~Schillings, S.~Schmidt, and V.~Schulz}, {\em Efficient shape
  optimization for certain and uncertain aerodynamic design}, Computers \&
  Fluids, 46 (2011), pp.~78--87.
\newblock 10th ICFD Conference Series on Numerical Methods for Fluid Dynamics
  (ICFD 2010).

\bibitem{SZ19}
{\sc C.~Schwab and J.~Zech}, {\em Deep learning in high dimension: {N}eural
  network expression rates for generalized polynomial chaos expansions in
  {UQ}}, Anal. Appl., 17 (2019), pp.~19--55.

\bibitem{vanBarelVandewalle}
{\sc A.~{Van Barel} and S.~Vandewalle}, {\em {R}obust {o}ptimization of {PDE}s
  with {r}andom {c}oefficients {u}sing a {m}ultilevel {M}onte {C}arlo
  {m}ethod}, SIAM/ASA J. Uncertain. Quantif., 7 (2019), pp.~174--202.

\bibitem{YMK20}
{\sc L.~Yang, X.~Meng, and G.~E. Karniadakis}, {\em B-pinns: {B}ayesian
  physics-informed neural networks for forward and inverse {PDE} problems with
  noisy data}, J. Comput. Phys., 425 (2021), p.~109913.

\bibitem{Y17}
{\sc D.~Yarotsky}, {\em Error bounds for approximations with deep {R}e{LU}
  networks}, Neural Netw., 94 (2017), pp.~103--114.

\bibitem{Zech2018}
{\sc J.~Zech}, {\em Sparse-Grid Approximation of High-Dimensional Parametric
  PDEs}, dissertation, ETH Zurich, 2018.

\end{thebibliography}


\end{document}